\documentclass[a4paper,11pt,oneside,reqno]{amsart}
\usepackage{geometry}

\usepackage[english]{babel}
\usepackage{csquotes}

\usepackage{lmodern}
\usepackage{amsmath,amssymb}
\usepackage{mathtools}
\usepackage{bm}
\usepackage{stmaryrd,MnSymbol}
\usepackage{relsize}

\usepackage{xcolor}
\usepackage{enumitem}
\usepackage{xpatch}
\usepackage{keytheorems} 
\usepackage[textsize=tiny]{todonotes}

\usepackage{setspace}
\setdisplayskipstretch{}

\definecolor{linkblue}{RGB}{1,1,190}
\definecolor{citered}{RGB}{190,1,1}

\usepackage[linkcolor=linkblue
           ,urlcolor=linkblue
           ,citecolor=citered
           ,colorlinks
           ,bookmarksopen=True,
           ]{hyperref}

\makeatletter
  \AtBeginDocument{
    \hypersetup{
      pdftitle  = {\@title},
    }
  }
\makeatother
  
\usepackage[capitalize]{cleveref}

\newkeytheorem{theorem}[numberwithin=section]
\newkeytheorem{
  definition,
  lemma,
  proposition,
  corollary,
  question}
[numberlike=theorem]

\newkeytheorem{
  remark,
  remarks,
  example,
  examples
}[numberlike=theorem,style=remark]

\setlist[enumerate,1]{label=\textup{(\arabic*)}, ref=\textup{(}\arabic*\textup{)},
  itemsep=0.5em plus 0.15em minus 0.05em,
  topsep=0.5em plus 0.15em minus 0.05em,
  leftmargin=0.75cm}
\setlist[enumerate,2]{label=\textup{(\roman*)}, ref=\textup{(}\roman*\textup{)},
  itemsep=0.5em plus 0.15em minus 0.05em,
  topsep=0.5em plus 0.15em minus 0.05em}
\setlist[itemize, 1]{itemsep=0.5em plus 0.15em minus 0.05em,
  topsep=0.5em plus 0.15em minus 0.05em, leftmargin=0.75cm}

\newlist{equivenumerate}{enumerate}{1}
\setlist[equivenumerate,1]{%
  label=\textup{(\alph*)},
  ref=\textup{(}\alph*\textup{)},
  itemsep=0.5em plus 0.15em minus 0.05em,
  topsep=0.5em plus 0.15em minus 0.05em,
  leftmargin=0.75cm
}

\xpatchcmd{\paragraph}{\normalfont}{{\normalfont\bfseries}}{}{}

\usepackage[backend=biber,
            style=alphabetic,
            giveninits=true,
            url=false,
            isbn=false,
            maxnames=99]
           {biblatex}
\renewbibmacro{in:}{}
\DeclareFieldFormat[article,inbook,incollection,inproceedings,patent,thesis,unpublished]{title}{{#1\isdot}}
\DeclareFieldFormat[article,inbook,incollection,inproceedings,patent,thesis,unpublished]{note}{\textit{#1\isdot}}
\DeclareFieldFormat[article,inproceedings,incollection]{pages}{#1}
\DeclareFieldFormat{postnote}{#1} 

\renewbibmacro*{issue+date}{%
  \iffieldundef{year}{}{ 
  \printtext[parens]{%
    \printfield{issue}%
    \setunit*{\addspace}%
    \usebibmacro{date}}%
  \newunit}}

\AtBeginBibliography{\small}

\usepackage{microtype}

\newcommand{\defit}[1]{\textsf{#1}}

\newcommand{\N}{\mathbb N}
\newcommand{\Z}{\mathbb Z}
\newcommand{\Q}{\mathbb Q}

\newcommand{\C}{\mathbb C}

\newcommand{\cA}{\mathcal A}
\newcommand{\cB}{\mathcal B}
\newcommand{\cG}{\mathcal G}
\newcommand{\cH}{\mathcal H}
\newcommand{\cL}{\mathcal L}
\newcommand{\cM}{\mathcal M}
\newcommand{\cX}{\mathcal X}

\newcommand{\mahler}{\mathcal M}

\newcommand{\puipol}[2]{{#1}[{#2}^{\frac{1}{\infty}}]}
\newcommand{\puiser}[2]{{#1}(\!({#2}^{\frac{1}{\infty}})\!)}

\newcommand{\power}[1]{{#1}\llbracket x \rrbracket}
\newcommand{\ru}{\mu}  
\newcommand{\rucommon}{\mu_{[k]}} 

\newcommand{\algc}[1]{\overline{#1}}

\newcommand{\md}{\mathfrak d}

\DeclareMathOperator{\End}{End}
\DeclareMathOperator{\ann}{ann}
\DeclareMathOperator{\val}{\mathsf v}

\title{Mahler series with multiplicative coefficient sequences}

\author{Jason Bell}
\address{Department of Pure Mathematics\\ University of Waterloo\\ Waterloo, ON\\ Canada N2L 3G1}
\email{jpbell@uwaterloo.ca}

\author{Daniel Smertnig}
\address{Faculty of Mathematics and Physics (FMF)\\
  University of Ljubljana
  and Institute of Mathematics, Physics and Mechanics (IMFM)\\
  Jadranska ulica 21\\
  1000 Ljubljana, Slovenia}
\email{daniel.smertnig@fmf.uni-lj.si}

\thanks{Bell was supported by NSERC grant RGPIN-2022-02951. Smertnig was supported by the Slovenian Research and Innovation Agency (ARIS) program P1-0288 and grant J1-60025.}

\subjclass[2020]{Primary 11B85; Secondary 11N64, 39A06}
\keywords{Mahler series, Mahler functions, multiplicative sequences, regular sequences, automatic sequences}

\begin{document}

\begin{abstract}
  \begin{singlespace}
    We prove that every Mahler series, over a field of characteristic $0$, with multiplicative coefficients is regular in the sense of Allouche and Shallit.
    We also obtain an explicit characterization of such series.
    This yields a joint extension of the characterization of rational series with multiplicative coefficients (by Bézivin and Bell--Bruin--Coons) and of multiplicative automatic sequences (by Konieczny–Lemańczyk–Müllner).
    Both of these results are used in our characterization, so we do not obtain new proofs of these special cases.
  \end{singlespace}
\end{abstract}

\date{}

\maketitle

\section{Introduction}

A sequence $f\colon \N \to K$, over a field $K$ of characteristic $0$, is \defit{multiplicative} if $f(mn)=f(m)f(n)$ whenever $m$ and $n$ are coprime.
Multiplicative sequences are ubiquitous in number theory, for instance, the Möbius function, the Euler totient function, and the divisor function are multiplicative.
However, if a multiplicative sequence $f$ additionally satisfies algebraic-combinatorial properties, it has a much more restricted behavior.

For instance, if the generating function of $f$ is algebraic, there is the following characterization.

\begin{theorem}[{\cite[Theorem 1.5]{bell-bruin-coons12}}] \label{t:mult-lrs}
  Let $K$ be a field of characteristic $0$.
  Let $F = \sum_{n=0}^\infty f(n) x^n \in \power{K}$ be an algebraic series over $K(x)$.
  If $f\colon \N \to K$ is multiplicative, then there exist $r \ge 0$ and an eventually periodic multiplicative function $\chi\colon \N \to K$ such that
  \[
  f(n) = n^r \chi(n) \qquad\text{for all } n \ge 1.
  \]
\end{theorem}

Every eventually periodic multiplicative function $\chi$ is easily seen to be eventually zero or to be periodic (\cref{l:mult-eventually-periodic}).

\Cref{t:mult-lrs} was proved by Bézivin \cite{bezivin95} for $K=\mathbb C$.
The result for arbitrary fields of characteristic $0$ is due to Bell, Bruin, and Coons \cite{bell-bruin-coons12}.
For $K=\C$, the theorem also holds under the weaker assumption that $F$ is $D$-finite \cite[Theorem 1.6]{bell-bruin-coons12}.
\Cref{t:mult-lrs} applies in particular to the case where $F$ is rational, or, equivalently, the coefficients $f$ form a linear recurrence sequence (LRS).

In a different direction, a sequence $f$ is $k$-automatic (for some integer $k \ge 2$) if there exists a finite state automaton that computes $f(n)$ when given as input the base-$k$ expansion of $n$ \cite{allouche-shallit92,allouche-shallit03}. 
The following result of Konieczny, Lemańczyk, and Müllner characterizes the multiplicative $k$-automatic sequences.

\begin{theorem}[\cite[Theorem 1.1]{konieczny-lemanczyk-muellner22}] \label{t:mult-auto}
  Let $L$ be a field \textup(of arbitrary characteristic\textup) and $f\colon \N \to L$ a multiplicative $k$-automatic sequence.
  Then $f$ is $p$-automatic for some prime $p$ and takes the form
  \begin{equation} \label{eq:mult-auto}
  f(p^i m) = g(i) \chi(m) \qquad\text{for all $i \ge 0$ and $m \in \N$ with $p \nmid m$},
  \end{equation}
  where $g\colon \Z_{\ge 0} \to L$ is eventually periodic with $g(0)=1$, and $\chi\colon \N \to L$ is a multiplicative eventually periodic function.

  Furthermore, any sequence given by \eqref{eq:mult-auto} with these conditions is multiplicative and $p$-automatic and this decomposition is unique unless $f$ is eventually periodic.
\end{theorem}

A $k$-automatic sequence is a special case of a $k$-regular sequence (in the sense of Allouche and Shallit), which is in turn a special case of the coefficient sequence of a $k$-Mahler series (we recall the notions in \cref{sec:preliminaries}).
The class of $k$-Mahler series is a natural generalization of rational series that has recently been extensively studied, see for instance \cite{adamczewski-faverjon17,roques18,adamczewski-faverjon18,bell-chyzak-coons-dumas19,faverjon-poulet22,adamczewski-bell-smertnig23,adamczewski-dreyfus-hardouin-wibmer24,adamczewski-faverjon24,faverjon-roques24,faverjon-roques26}.

The class of $k$-Mahler series is essentially disjoint from the class of $D$-finite series, in that the only series that are both $k$-Mahler and $D$-finite are the rational series \cite{bezivin94,bell-coons-rowland13}.
In light of this, it is natural to ask for a characterization of $k$-Mahler series with multiplicative coefficients.
This is the main result of the present paper, and constitutes a joint extension of the rational series case of \cref{t:mult-lrs} and of \cref{t:mult-auto}.

\begin{theorem} \label{t:mult-mahler-main}
  Let $K$ be a field of characteristic $0$, let $k \ge 2$, and let $F = \sum_{n=0}^\infty f(n) x^n$ be a $k$-Mahler series.
  If $f\colon \N \to K$ is multiplicative, then $f$ is $k$-regular.
  Further, there exist a prime $p$, a linear recurrence sequence $g\colon \Z_{\ge 0} \to K$ with $g(0)=1$, an integer $r \ge 0$, and a multiplicative eventually periodic function $\chi\colon \N \to K$ such that
  \begin{equation} \label{eq:explicit-formula}
    f(p^i m) = g(i) m^r \chi(m) \qquad\text{for all $i \ge 0$ and $m \in \N$ with $p \nmid m$}.
  \end{equation}
\end{theorem}

Observe that \eqref{eq:explicit-formula} means that $f$ splits into a product of two multiplicative functions: the first one, namely $p^i m \mapsto g(i)$, is non-trivial (that is, not identical to $1$) only on powers of $p$; the second one, given by $p^i m \mapsto m^r \chi(m)$, is non-trivial only on integers coprime to $p$.
The first factor is $p$-regular, but in general not an LRS, while the second factor is both $k$-regular and an LRS.

It is easy to see that every sequence $f$ as in  \eqref{eq:explicit-formula} is multiplicative and $p$-regular, so \cref{t:mult-mahler-main} gives a full characterization.
Aside from obvious obstructions, the representation is also unique (see \cref{p:uniqueness}).

The theorem really splits into two cases: if $k$ is not a prime power, then in fact $f$ is an LRS, and there is the representation $f(n) = n^r \chi(n)$ for all $n \ge 1$.
A representation of the form as in \eqref{eq:explicit-formula} then holds for every prime $p$ (\cref{l:mult-rational-prime-repr}).
In particular, we have the following.

\begin{corollary} \label{c:mult-mahler-nonprime}
  Let $K$ be a field of characteristic $0$, let $k \ge 2$, and let $F = \sum_{n=1}^\infty f(n) x^n$ be a $k$-Mahler series.
  If $k$ is not a prime power, then $F$ is rational.
\end{corollary}

On the other hand, if $k=p^e$, then \eqref{eq:explicit-formula} holds for the particular prime $p$.
If $f$ is not an LRS, then $p$ is the unique such prime (by Adamczewski and Bell's generalization of Cobham's theorem \cite{adamczewski-bell17}, as otherwise $f$ is $p$- and $q$-regular for two distinct primes; see also \cite{schaefke-singer19}).

To consider some examples, the sequence of $p$-adic valuation $f(n) = \val_p(n)$ is multiplicative and $p$-regular but not an LRS.
Here only the first factor of \eqref{eq:explicit-formula} is non-trivial.
The function $f(n) = n / 2^{\val_2(n)}$ is multiplicative and $2$-regular but not an LRS. 
In this example both factors of \eqref{eq:explicit-formula} are non-trivial.

As a consequence of \cref{t:mult-mahler-main}, it is easy to see that various multiplicative functions arising in number theory are not $k$-Mahler for any $k \ge 2$, for instance, the divisor function $\sigma_l(n) = \sum_{d \mid n} d^l$, the Euler totient function, and more generally Jordan's totient function.
(For finitely valued multiplicative functions, such as the Möbius function, this is already clear from \cref{t:mult-auto} together with the fact that every finitely-valued $k$-Mahler series is $k$-automatic \cite[Theorem 11.1]{adamczewski-bell-smertnig23}.)

The proof strategy is as follows.
We first prove the characterization in the main result under the stronger hypothesis that $F$ is $k$-regular, in \cref{sec:regular-case}.
The tools here are the Konieczny--Lemańczyk--Müllner characterization of multiplicative automatic sequences, and a lifting argument modulo maximal ideals of a finitely generated ring to the regular case \cite[Theorem 3.11]{bell05-cobham}.

It then remains to show that if $F$ is $k$-Mahler with multiplicative coefficients, then $F$ is $k$-regular.
This splits into two cases.
In the first case there are arbitrarily large primes $q$ such that $f(qn) \ne f(q) f(n)$ for some $n \in \N$.
This case is treated in \cref{sec:mult-mahler-non-compl-mult}. 
The key tools here are the classification of $k$-regularity of Mahler series in terms of the Mahler denominator \cite[Theorem 10.4]{adamczewski-bell-smertnig23}, and properties of the Mahler denominator that are established in \cref{sec:mahler-denominator} (and whose proofs themselves use the results from \cite{adamczewski-bell-smertnig23}).

For the second case we first show some algebraic results on Mahler operators in \cref{sec:rings-mahler-operators}.
The proof, in \cref{sec:mult-mahler-completely-multiplicative}, then uses the minimal inhomogeneous Mahler equation for $F$ and again the characterization of $k$-regularity from \cite{adamczewski-bell-smertnig23}.

The two cases are in fact not disjoint: in the second case, it suffices to have one sufficiently large prime $q$ such that $f(qn) = f(q)f(n)$ for all $n \in \N$.

\smallskip
\paragraph{Notation} Throughout the paper, let $K$ be a field of characteristic $0$, let $k \ge 2$ be an integer, and let $K \subseteq \algc{K}$ be an algebraic closure.
By $\N=\{1,2,3,\dots\}$ we denote the positive integers.
We write $\mu_n(K)$ for the $n$-th roots of unity in $K$, write $\mu_n^*(K)$ for the primitive $n$-th roots of unity, and write $\ru(K)$ for all roots of unity.
Finally, we write $\rucommon(K)$ for the set of all roots of unity in $K$ whose order has a common factor with $k$.
Note that $\zeta \in \rucommon(K)$ if and only if $\zeta$ is a root of unity with $\zeta^{k^i} \ne \zeta$ for all $i \ge 1$.

\section{Preliminaries} \label{sec:preliminaries}

We recall the special classes of formal power series that appear in the paper: rational series, Mahler series, regular series, and automatic series.
References for rational series are \cite[Chapter 6]{berstel-reutenauer11}\cite{everest-vanderpoorten-shparlinski-ward03}, for regular and automatic series \cite{allouche-shallit03}\cite[Chapter 5]{berstel-reutenauer11}, and for Mahler series \cite{dumas93,nishioka96,adamczewski-faverjon24}.

Usually sequences are indexed starting from $0$, that is $f\colon \Z_{\ge 0} \to K$, but for multiplicative sequences it is more natural to start from $1$, that is $f\colon \N \to K$.
Since all the classes of series that we consider are closed under changing finitely many coefficients, this does not cause any issues.
\subsection{Rational series and linear recurrence sequences}

Let $F= \sum_{n=0}^\infty f(n) x^n \in \power{K}$  be a formal power series with coefficient sequence $f\colon \Z_{\ge 0} \to K$.
Then the following statements are equivalent:
\begin{itemize}
  \item The series $F$ is \defit{rational}, that is, of the form $F=P/Q$ with $P$, $Q \in K[x]$ and $Q(0) \ne 0$. 
  \item The sequence $f$ is a \defit{linear recurrence sequence} (\defit{LRS}), that is, of the form $f(n) = \sum_{i=1}^d c_i f(n-i)$ for some $d \ge 0$, coefficients $c_1$, \dots,~$c_d \in K$, and all $n \ge d$.
  \item The sequence $f$ has a \defit{linear representation}: there exist $n \ge 0$, a matrix $A \in K^{n \times n}$, and two vectors $u \in K^{1 \times n}$ and $v \in K^{n \times 1}$ such that $f(n) = u A^n v$ for all $n \ge 0$.
\end{itemize}

The class of LRS is closed under $K$-linear combinations, Hadamard products (for which $(fg)(n) = f(n)g(n)$), and Cauchy products (for which $(f \ast g)(n) = \sum_{m=0}^n f(m)g(n-m)$).
The following lemma, that we will later need, can be obtained using the linear representation characterization of LRS.

\begin{lemma} \label{l:rational-arith-progr}
  Let $a \ge 1$.
  A sequence $f\colon \mathbb Z_{\ge 0} \to K$ is an LRS if and only if $n \mapsto f(an+b)$ is an LRS for all $0 \le b \le a-1$.
\end{lemma}

\subsection{Mahler series}
Mahler series, for the fixed base $k \ge 2$, are defined in terms of a functional equation involving the Mahler operator $\cM_k\colon \power{K} \to \power{K}$, $F(x) \mapsto F(x^k)$.

\begin{definition}
A formal power series $F = \sum_{n=0}^\infty f(n) x^n$ is \defit{$k$-Mahler} \textup(also called an \defit{$M$-function}\textup) if there exist $n \ge 1$ and polynomials $P_0$, \dots,~$P_n \in K[x]$, with $P_0 \ne 0$, such that
\begin{equation} \label{eq:mahler-eqn}
P_0(x) F(x) = \sum_{j=1}^n P_j(x) F(x^{k^j}).
\end{equation}
\end{definition}

An equation of the form \eqref{eq:mahler-eqn} is called a \defit{\textup(homogeneous\textup) $k$-Mahler equation} for $F$.
The integer $n$ is the \defit{order} of the equation.
Allowing for inhomogeneous equations does not enlarge the class of series, and we will usually only consider homogeneous equations, except in \cref{sec:mult-mahler-completely-multiplicative}.

It is easy to see that the polynomials appearing as coefficient $P_0$ in $k$-Mahler equations for a fixed series $F$ form an ideal $I_F$ of $K[x]$.
Since $K[x]$ is a principal ideal domain (PID), this ideal is generated by a single polynomial.

\begin{definition}
  The \defit{$k$-Mahler denominator} of a $k$-Mahler series $F$, denoted by $\md_K F \in K[x]$, is the unique generator of $I_F$ whose lowest nonzero coefficient is $1$.
\end{definition}

When the base $k$ is clear from context, we simply write \defit{Mahler denominator}.
By definition, the Mahler denominator $\md_k F$ divides $P_0$ in any $k$-Mahler equation for $F$.
Further, there exists a $k$-Mahler equation for $F$ with $P_0=\md_k F$.
However, the polynomial $P_0$ in a minimal Mahler equation, that is, one of minimal order, need not be the Mahler denominator \cite[Example 3.10]{adamczewski-bell-smertnig23}.

Mahler series are closed under taking $K$-linear combinations \cite[Théorème 3 in Chapitre 3]{dumas93}.
Every rational series is $k$-Mahler, as the next example shows.

\begin{example} \label{exm:rational-is-mahler}
  If $F=P/Q$ with $P$, $Q \in K[x]$ and $Q(0) \ne 0$, then $F$ is $k$-Mahler.
  Indeed, from $Q(x) F(x) = P(x)$ and $Q(x^k) F(x^k) = P(x^k)$ we get the $k$-Mahler equation
  \[
  P(x^k) Q(x) F(x) = P(x) Q(x^k) F(x^k).
  \]
\end{example}

Among $k$-Mahler series, the rational series can be considered to be the trivial case.

Two integers $k$,~$l \ge 2$ are \defit{multiplicatively independent} if there do not exist integers $m$,~$n \ge 1$ such that $k^m = l^n$.
The following generalization of Cobham's Theorem shows that the class of $k$- and $l$-Mahler series are disjoint except for the rational series.

\begin{theorem}[{\cite[Theorem 1.3]{adamczewski-bell17}}]\label{t:cobham-mahler}
  Let $k$,~$l \ge 2$ be multiplicatively independent integers.
  Then $F \in \power{K}$ is both $k$- and $l$-Mahler if and only if it is rational.
\end{theorem}

\subsection{Regular sequences}

To define regular sequences, we first introduce the $k$-kernel.

\begin{definition}
  Let $f \colon \Z_{\ge 0} \to S$ be a sequence over a set $S$.
  The \defit{$k$-kernel} of $f$ is the set of subsequences
  \[
  \big\{\, (f(k^e n + r))_{n \ge 0} : e \ge 0,\, 0 \le r < k^e \,\}.
  \]
\end{definition}

If $f$ takes its values in a commutative ring $R$, then one can consider the $R$-module generated by the $k$-kernel of $f$.
This leads to the notion of $k$-regular sequences, introduced by Allouche and Shallit \cite{allouche-shallit92,allouche-shallit03}.

\begin{definition}
  Let $R$ be a commutative ring and $f\colon \Z_{\ge 0} \to R$ a sequence.
  Then $f$ is \defit{$k$-regular} \textup(over $R$\textup) if the $R$-module generated by its $k$-kernel is finitely generated.
\end{definition}

We are interested in particular in the case that $R$ is a field $K$, but will need the more general case in the proofs.
A power series $F = \sum_{n=0}^\infty f(n) x^n \in \power{R}$ is \defit{$k$-regular} if its coefficient sequence $f$ is $k$-regular.

The analogue of \cref{l:rational-arith-progr} holds for regular sequences.

\begin{lemma} \label{l:regular-arith-progr}
  Let $a \ge 1$.
  A sequence $f\colon \mathbb Z_{\ge 0} \to R$ is $k$-regular if and only if $n \mapsto f(an+b)$ is $k$-regular for all $0 \le b \le a-1$.
\end{lemma}

\begin{proof}
  \cite[Theorems 2.6 and 2.7]{allouche-shallit92}.
\end{proof}

Equivalently, a sequence $f$ is $k$-regular if there exist an integer $d \ge 1$, vectors $u \in R^{1 \times d}$, $v \in R^{d \times 1}$, and matrices $A(0)$, \dots,~$A(k-1) \in R^{d \times d}$ such that
\begin{equation} \label{eq:lin-rep}
f(n_l k^l + n_{l-1} k^{l-1} + \cdots + n_0) = u  A(n_l) A(n_{l-1}) \cdots A(0) v
\end{equation}
for all $l \ge 0$ and $0 \le n_l$, \dots,~$n_0 \le k-1$.

From the representation in \eqref{eq:lin-rep} it is immediate that for every $k$-regular sequence $f \colon \Z_{\ge 0} \to K$, there exists a finitely generated subring $R \subseteq K$ over which $f$ is $k$-regular.
Namely, one can take the subring generated by the entries of the matrices $A(i)$ and the vectors $u$ and $v$.
We shall later also need the following related observation.

\begin{lemma} \label{l:regular-descent-field}
  If $F \in \power{K}$ is $k$-regular over $\overline{K}$, then $F$ is $k$-regular over $K$.
\end{lemma}

\begin{proof}
  By the argument above, we see that $F$ is $k$-regular over some finitely generated subfield $L \subseteq \overline{K}$ containing $K$.
  Because $L$ is algebraic, the field extension $L/K$ is finite.
  Now the $k$-kernel is finite-dimensional over $L$, and hence also over $K$, so $F$ is $k$-regular over $K$.
\end{proof}

If $F \in \power{K}$ is $k$-regular, then it is $k$-Mahler \cite{becker94}.
For a statement in the converse direction, we first make the following definitions.

\begin{definition}
  A polynomial $P \in K[x]$ is 
  \begin{itemize}
    \item \defit{negligible} \textup(more precisely, it is \defit{$k$-negligible}\textup) if all roots of $P$ in $\algc{K}$ are contained in $\rucommon(\algc{K}) \cup \{0\}$;
    \item \defit{non-negligible} if it is not negligible.
  \end{itemize}
\end{definition}

In other words, for a negligible polynomial, all nonzero roots $\alpha$ are roots of unity, and $\alpha^{k^i} \ne \alpha$ for all $i \ge 1$.
This allows for the following characterization.

\begin{theorem}[\cite[Theorem 10.4]{adamczewski-bell-smertnig23}] \label{t:abs-regular}
  A $k$-Mahler series $F \in \power{K}$ is $k$-regular if and only if the Mahler denominator $\md_k F$ is negligible.
\end{theorem}

\subsection{Automatic sequences}
The final class of sequences we need are the automatic sequences.
They can be defined in terms of a finite state automaton, but also, equivalently, in terms of the $k$-kernel.

\begin{definition}
  A sequence $f\colon \Z_{\ge 0} \to S$ over a set $S$ is \defit{$k$-automatic} if its $k$-kernel is finite.  
\end{definition}

Again, a series $F=\sum_{n=0}^\infty f(n) x^n \in \power{K}$ is \defit{$k$-automatic} if its coefficient sequence $f$ is $k$-automatic.
It is clear that a $k$-automatic sequence is $k$-regular and hence $k$-Mahler.

Automatic sequences are always finitely valued, and this characterizes them among $k$-Mahler series.

\begin{theorem}[\cite[Theorem 11.1]{adamczewski-bell-smertnig23}]
  A $k$-Mahler series $F \in \power{K}$ is $k$-automatic if and only if it has only finitely many distinct coefficients.
\end{theorem}

Cobham's Theorem, in its original form, is the following statement about automatic sequences.

\begin{corollary}[Cobham's Theorem \cite{cobham69}] \label{c:cobham}
  Let $k$,~$l \ge 2$ be multiplicatively independent integers.
  Then a sequence $f\colon \Z_{\ge 0} \to S$ over a set $S$ is both $k$- and $l$-automatic if and only if it is eventually periodic.
\end{corollary}

This also follows from \cref{t:cobham-mahler} by embedding $S$ into a field and observing that a finitely-valued LRS is eventually periodic.
All the classes we have introduced have the following property in regard to changing the base $k$.

\begin{lemma}
  Let $F = \sum_{n=0}^\infty f(n) x^n \in \power{K}$ and let $r \ge 1$.
  \begin{enumerate}
  \item The series $F$ is $k$-Mahler if and only if it is $k^r$-Mahler.
  \item The series $F$ is $k$-regular if and only if it is $k^r$-regular.
  \item The series $F$ is $k$-automatic if and only if it is $k^r$-automatic.
  \end{enumerate}
\end{lemma}

Konieczny, Lemańczyk, and Müllner's characterized multiplicative automatic sequences \cite{konieczny-lemanczyk-muellner22}.
Their result is stated for the field of complex numbers, but the version for arbitrary fields that we need (\cref{t:mult-auto}) easily follows by embedding the (finitely generated) image of $f$ into $\C$ while preserving multiplication, using the next lemma.

\begin{lemma}
  If $L$ is a field and $S \subseteq (L,\cdot)$ is a finitely generated subsemigroup, then there exists an injective semigroup homomorphism $\varphi\colon S \to (\C,\cdot)$.
\end{lemma}

\begin{proof}
  Let $G \subseteq L^\times$ be the subgroup generated by $S \setminus \{0\}$.
  Then $G$ is a finitely generated abelian group, and hence isomorphic to $T \times \Z^m$ for some finite abelian group $T$.
  Since $L$ is a field, the group $T$ is cyclic, say $T \cong \Z/n\Z$ for some $n \ge 1$.
  Taking $p_1$, \dots,~$p_m$ to be distinct primes, we get an isomorphism $G \to \mu_n(\C) \times \langle p_1 \rangle \times \cdots \times \langle p_m \rangle \subseteq (\C,\cdot)$.
  Mapping $0$ to $0$ (if $0 \in S$) gives the desired homomorphism.
\end{proof}

\section{Multiplicative \texorpdfstring{$k$}{k}-regular sequences} \label{sec:regular-case}

In this section we characterize $k$-regular series $F \in \power{K}$ with multiplicative coefficients.
This results in a first version of our main theorem in \cref{p:explicit-form}, under the stronger hypothesis that the series is not just $k$-Mahler but even $k$-regular.
First we consider the case in which $k$ is not a prime power.

\begin{proposition} \label{p:regular-distinct-primes}
  Suppose that $k$ has at least two distinct prime divisors.
  If $F \in \power{K}$ is a $k$-regular series with multiplicative coefficients, then $F$ is rational.
\end{proposition}

\begin{proof}
  Since $F$ is $k$-regular, there exists a finitely generated $\Z$-algebra $R \subseteq K$ such that $F \in R\llbracket x \rrbracket$ and with $F$ being $k$-regular over $R$.
  By \cite[Theorem 3.11]{bell05-cobham}, it suffices to show that the reduction $\overline{F} \in R/M\llbracket x \rrbracket$ has eventually periodic coefficients for every maximal ideal $M$ of $R$.

  Let $M$ be a maximal ideal of $R$. Then $R/M$ is a finite field (by Zariski's Lemma), and therefore $\overline{F} \in R/M\llbracket x \rrbracket$ is $k$-regular and finitely valued, hence $k$-automatic.
  But $\overline{F}$ also has multiplicative coefficients. 
  Konieczny, Lemańczyk, and Müllner's characterization of multiplicative automatic sequences (\cref{t:mult-auto}) shows that $\overline{F}$ is $p$-automatic for some prime $p$.
  Thus, the series $\overline{F}$ is both $k$- and $p$-automatic, and $k$ and $p$ are multiplicatively independent, because $k$ has at least two distinct prime divisors.
  Now Cobham's Theorem (\cref{c:cobham}) implies that $\overline{F}$ has eventually periodic coefficients.
\end{proof}

Multiplicative coefficient sequences of rational series are characterized by \cref{t:mult-lrs}.
To obtain a unified result in all cases, it is useful to observe that they also possess the following representation.

\begin{lemma} \label{l:mult-rational-prime-repr}
  Let $F = \sum_{n=1}^\infty f(n) x^n \in K\llbracket x \rrbracket$ be a rational series with multiplicative coefficients and let $p$ be a prime number.
  Then there exist $r \ge 0$ and a multiplicative eventually periodic function $\chi\colon \N \to K$ such that
  \[
  f(p^i m) = p^{ri} \chi(p^i) \cdot m^r \chi(m) \qquad\text{for all $i \ge 0$ and $m \in \N$ with $p \nmid m$}.
  \]
  Moreover, the sequence $g \colon \Z_{\ge 0} \to K$, $i \mapsto p^{ri} \chi(p^i)$ is a linear recurrence sequence.
\end{lemma}

\begin{proof}
  By \cref{t:mult-lrs}, there exist $r \ge 0$ and a multiplicative eventually periodic function $\chi\colon \N \to K$ such that $f(n) = n^r \chi(n)$ for all $n \ge 1$.
  Substituting $n = p^i m$ with $p \nmid m$ gives the desired representation.
  Since $i \mapsto p^{ri}$ is an LRS, and $i \mapsto \chi(p^i)$, being eventually periodic, is trivially an LRS, so is their product sequence $g$.
\end{proof}

We now deal with the more interesting case in which $k$ is a prime power.

\begin{lemma} \label{l:regular-single-prime}
  Suppose $k=p^e$ with $p$ prime and $e \ge 1$.
  If $F = \sum_{n=1}^\infty f(n) x^n \in K\llbracket x \rrbracket$ is a $k$-regular series with multiplicative coefficients, then there exists a rational series $H \in \power{K}$ such that 
  \begin{equation} \label{eq:decomp-f-h}
  F(x) = \sum_{i=0}^\infty f(p^i) H(x^{p^i}).
  \end{equation}
\end{lemma}

\begin{proof}
  Define $h\colon \N \to K$ by 
  \[
  h(n) = \begin{cases}
    f(n) & \text{if } p \nmid n,\\
    0    & \text{if } p \mid n,
  \end{cases}
  \]
  and let $H = \sum_{n=1}^\infty h(n) x^n \in K\llbracket x \rrbracket$.
  Then $h$ is $k$-regular by \cref{l:regular-arith-progr}, and it is clearly multiplicative.
  The multiplicativity of $f$ implies the decomposition \eqref{eq:decomp-f-h}.

  It remains to show that $H$ is rational. 
  As in the proof of \cref{p:regular-distinct-primes}, there exists a finitely generated $\Z$-algebra $R \subseteq K$ such that $H \in R\llbracket x \rrbracket$ and $H$ is $k$-regular over $R$, and it suffices to show that $\overline{H} \in R/M\llbracket x \rrbracket$ has eventually periodic coefficients for every maximal ideal $M$ of $R$ \cite[Theorem 3.11]{bell05-cobham}.

  Let $M$ be a maximal ideal of $R$.
  Then $\overline{H} \in R/M\llbracket x \rrbracket$ is $p$-regular and finitely valued, hence $p$-automatic.
  Since $\overline{H}$ is $p$-automatic and  multiplicative, again applying \cref{t:mult-auto} implies that $\overline{H}$ is $q$-automatic for some prime $q$.

  If $p \ne q$, then Cobham's Theorem (\cref{c:cobham}) implies that $\overline{H}$ has eventually periodic coefficients.
  If $p=q$, then \cref{t:mult-auto} shows 
  \[
    h(n) \equiv g_{p}(\val_p(n)) \cdot \chi_p\bigg(\frac{n}{p^{\val_p(n)}}\bigg) \mod M,
  \]
  with $g_{p}\colon \N \to R/M$ eventually periodic and $\chi_p\colon \N \to R/M$ multiplicative and eventually periodic.
  Since $h(n)=0$ for $p \mid n$, we get 
  \[
  h(n) \equiv
  \begin{cases}
    g_p(0) \chi_p(n) & \mod M \text{ for $p \nmid n$,}\\
    0                & \mod M \text{ for $p \mid n$.}
  \end{cases}
  \]
  Thus, again $\overline{H}$ is eventually periodic.
\end{proof}

The following is the main result of this section, giving the explicit form of multiplicative $k$-regular sequences.

\begin{proposition} \label{p:explicit-form}
  Let $K$ be a field of characteristic $0$, let $k \ge 2$, and let $F=\sum_{n=1}^\infty f(n) x^n \in K\llbracket x \rrbracket$ be a $k$-regular series.
  If $f$ is multiplicative, then there exists a linear recurrence sequence $g\colon \Z_{\ge 0} \to K$ with $g(0)=1$, an integer $r \ge 0$, a prime $p$, and a multiplicative eventually periodic function $\chi\colon \N \to K$ such that
  \begin{equation}
    f(p^i m) = g(i) \cdot m^r \chi(m) \qquad\text{for all $i \ge 0$ and $m \in \N$ with $p \nmid m$}.
  \end{equation}
\end{proposition}

\begin{proof}
  If $k$ is not a prime power, then \cref{p:regular-distinct-primes} implies that $F$ is rational, and the result follows from \cref{l:mult-rational-prime-repr} (for an arbitrary prime $p$).

  Suppose now that $k=p^e$ for some prime $p$.
  \Cref{l:regular-single-prime} implies that $F$ admits a decomposition as in \eqref{eq:decomp-f-h}, with $H \in \power{K}$ rational.

  \Cref{t:mult-lrs} shows $h(n) = n^r \chi(n)$ for some integer $r \ge 0$ and some multiplicative eventually periodic function $\chi\colon \N \to K$.
  Defining $g\colon \Z_{\ge 0} \to K$ by $g(i) \coloneqq f(p^i)$, now \eqref{eq:decomp-f-h} shows
  \[
  f(n) = g(\val_p(n)) \cdot \bigg(\frac{n}{p^{\val_p(n)}}\bigg)^r \cdot \chi\bigg(\frac{n}{p^{\val_p(n)}}\bigg)
  \]
  for all $n \ge 1$.

  If $f=0$, we can clearly take $g(0)=1$ by taking $\chi=0$.
  If $f \ne 0$, then the multiplicativity of $f$ and $\chi$ implies $f(1)=1$ and $\chi(1)=1$.
  Now $g(0)=f(1)=1$.
  Finally, since $f$ is $p$-regular, it admits a linear representation as in \eqref{eq:lin-rep}. 
  It is then immediate that the sequence $i \mapsto g(i) = f(p^i) = (u A(1)) A(0)^{i} v$ is an LRS.
\end{proof}

The obtained representation is essentially unique, as the following proposition shows.

\begin{proposition} \label{p:uniqueness}
  Suppose $f\colon \N \to K$ has a representation as in \eqref{eq:explicit-formula} in \cref{t:mult-mahler-main}.
  \begin{enumerate}
  \item \label{uniq:g} If $f \ne 0$, then $g$ is uniquely determined by $f$.
  \item \label{uniq:chi-r} If $f(n) \ne 0$ for infinitely many $n$ with $p \nmid n$, and $\chi$ is chosen such that $\chi(n)=0$ if $p \mid n$, then $r$ and $\chi$ are also uniquely determined by $f$.
  \end{enumerate}
\end{proposition}

\begin{proof}
  \ref{uniq:g} Since $f$ and $\chi$ are multiplicative and nonzero, it holds that $\chi(1)=f(1)=1$.
  We get $g(i) = f(p^i)$ for all $i \ge 0$.

  \smallskip
  \ref{uniq:chi-r}
  For $n \in \N$ with $p \nmid n$, we have $f(n) = n^r \chi(n)$.
  Since $f(n) \ne 0$ for infinitely many such $n$, necessarily $\chi(n) \ne 0$ infinitely often.
  Since $\chi$ is eventually periodic, we can find some $0 \ne c \in K$ and an infinite set $S$ of $n$ with $\chi(n)=c$.
  Then $n^r c = m^r c$ for all $m$,~$n \in S$ implies the uniqueness of $r$ (by the identity theorem for polynomials).
  With $r$ now fixed, the formula $f(n)=n^r \chi(n)$ also determines $\chi(n)$ for all $n$ with $p \nmid n$.
\end{proof}

\begin{remark}
The uniqueness result in \cref{p:uniqueness} is the best possible.
For \ref{uniq:g} this is obvious. For \ref{uniq:chi-r}, if $f(n)=0$ for all but finitely many $n \in \N$ with $p \nmid n$, then $\chi$ is eventually zero, and we can replace $\chi(n)$ by $\chi'(n)=\chi(n) n^s$ with any integer $s \le r$, to obtain $n^r \chi(n) = n^{r-s} \chi'(n)$.
\end{remark}

\section{Mahler denominators} \label{sec:mahler-denominator}

In preparation for dealing with the more general case of multiplicative $k$-Mahler series, we need to establish some properties of Mahler denominators.
In particular, in \cref{t:mahler-denominators}, we bound the Mahler denominator of a sum of two Mahler series.
To do so, we use the non-trivial results of \cite{adamczewski-bell-smertnig23} that characterize $k$-regular series among the $k$-Mahler series, see \cref{t:abs-regular} (we are not aware of a more direct proof).

\subsection{Mahler denominators of series in \texorpdfstring{$K\llbracket x^l \rrbracket$}{K[[x\^{}l]]}}

We start with some results on Mahler equations and denominators for series of the form $F(x) = \widetilde F(x^l)$.

In the formulation and the proof of the first lemma we use Cartier operators.
For each $l \ge 1$ and $0 \le r \le l-1$, the \defit{Cartier operator} of the residue class $r$ (with respect to the modulus $l$) is defined by 
\[
\Delta_r^{(l)} \colon K\llbracket x \rrbracket \to K\llbracket x \rrbracket, \quad \sum_{n=0}^\infty g(n) x^n \mapsto \sum_{n=0}^\infty g(ln + r) x^n
\]
Then $\Delta_r^{(l)} \circ \Delta_{r'}^{(l')} = \Delta_{l'r+r'}^{(ll')}$ for all $l$,~$l' \ge 1$ and $0 \le r \le l-1$, $0 \le r' \le l'-1$.
Moreover, for every $P \in K[x]$ and $G \in K\llbracket x \rrbracket$,
\[
\Delta_r^{(l)}\big(P(x) G(x^l)\big) = \Delta_r^{(l)}(P(x)) G(x) \qquad\text{and}\qquad \deg \Delta_r^{(l)}( P(x)) \le \frac{\deg P(x)}{l}.
\]
These are easy to check when $P$ is a monomial, and they extend to arbitrary polynomials by $K$-linearity of $\Delta_r^{(l)}$.
Observe also that if $P \in K[x]$ is nonzero and $l \ge 1$, it may happen that $\Delta_r^{(l)}(P)=0$ for some $r$, but there exists at least one $r$ for which $\Delta_r^{(l)}(P) \ne 0$.

\begin{lemma} \label{l:mahler-eqn-l-power}
  Let $F \in K\llbracket x \rrbracket$ be such that $F = \widetilde F(x^l)$ for some $l \ge 1$ and $\widetilde F\in K\llbracket x \rrbracket$.
  If
  \begin{equation} \label{eq:general-mahler-eqn}
  \sum_{i=0}^n P_i(x) F(x^{k^i}) = A(x),
  \end{equation}
  with $P_0$, \dots,~$P_n$, $A \in K[x]$ and $P_i \ne 0$ for some $i$, then
  there exist $Q_0$, \dots,~$Q_n$, $B \in K[x]$ with $Q_0 \ne 0$ such that
  \[
  \sum_{i=0}^n Q_i(x^l) F(x^{k^i}) = B(x^l).
  \]
  Moreover, if $i$ is minimal with $P_i \ne 0$, then we can take $Q_0 = \Delta_{r}^{(lk^i)}(P_i)$ for some $0 \le r \le lk^i-1$.
\end{lemma}

\begin{proof}
  Let $i \ge 0$ be minimal with $P_i \ne 0$.
  First suppose $i \ge 1$.
  Then there exists some $0 \le s \le k-1$ such that $\Delta_s^{(k)}(P_i) \ne 0$.
  Applying $\Delta_s^{(k)}$ to \eqref{eq:general-mahler-eqn} gives an equation of the same type, of order at most $n-1$, and with $F(x^{k^{i-1}})$ having nonzero coefficient.
  After at most $i$ iterations, which corresponds to an application of an operator $\Delta_{r'}^{(k^i)}$, we can assume $i=0$.
  
  Now suppose we have an equation of the form \eqref{eq:general-mahler-eqn} with $P_0 \ne 0$.
  Choose $0 \le r \le l-1$ such that $\Delta_r^{(l)}(P_0) \ne 0$.
  Substituting $F=\widetilde F(x^l)$ into the Mahler equation, and applying $\Delta_r^{(l)}$, gives
  \[
  \sum_{j=0}^n \Delta_r^{(l)} P_j(x) \widetilde F(x^{k^j}) = \Delta_r^{(l)} A(x).
  \]
  Substituting $x^l$ for $x$, we get
  \[
  \sum_{j=0}^n (\Delta_r^{(l)} P_j)(x^l) F(x^{k^j}) = (\Delta_r^{(l)} A)(x^l).
  \]
  Setting $Q_j \coloneqq \Delta_r^{(l)} P_j$ and $B \coloneqq \Delta_r^{(l)} A$, we are done.  
\end{proof}

\begin{lemma} \label{l:divide-support}
  Let $F \in K\llbracket x \rrbracket$ be $k$-Mahler of order $n$.
  Suppose that $F = \widetilde F(x^l)$ for some $l \ge 1$ and $\widetilde F\in K\llbracket x \rrbracket$.
  \begin{enumerate}
  \item \label{div-sup:eqn} The series $F$ satisfies a Mahler equation of the form
  \[
  Q_0(x^l) F(x) = \sum_{j=1}^n Q_j(x^l) F(x^{k^j})
  \]
  with $Q_0$, \dots,~$Q_n \in K[x]$ and $Q_0 \ne 0$.

  \item \label{div-sup:denom} The Mahler denominator of $F$ satisfies $(\md_k F)(x)=(\md_k \widetilde F)(x^l)$. In particular $\md_k F$ is a polynomial in $x^l$.
  \end{enumerate}
\end{lemma}

\begin{proof}
  \ref{div-sup:eqn} 
  Immediate from \cref{l:mahler-eqn-l-power}.
 
  \smallskip
  \ref{div-sup:denom}
  First note that if 
  \[
  Q_0(x) \widetilde F(x) = \sum_{j=1}^m Q_j(x) \widetilde F(x^{k^j})
  \]
  with $Q_0 \ne 0$ is a $k$-Mahler equation for $\widetilde F$, then substituting $x^l$ for $x$ gives a $k$-Mahler equation for $F$ with $Q_0(x^l)$ as lowest coefficient.
  This implies that $\md_k F$ divides $(\md_k \widetilde F)(x^l)$.

  Applying \cref{l:mahler-eqn-l-power} to a Mahler equation for $F$ in which $P_0=\md_k F$, it follows that $\md_k \widetilde F$ divides $\Delta_r P_0$.
  Now
  \[
  \deg P_0 = \deg \md_k F \,\le\, l \deg \md_k \widetilde F \,\le\, l \frac{\deg P_0}{l} = \deg P_0.
  \]
  Hence, equality holds throughout.
  Since $P_0$ divides $(\md_k \widetilde F)(x^l)$, we get $P_0 = \md_k \widetilde F(x^l)$.
\end{proof}

We will also need the following variant.

\begin{lemma} \label{l:frac-mahler-eqn}
  Let $F \in \power{K}$.
  Suppose that there exist $P_0$, \dots,~$P_n$, $A \in K[x]$, with some $P_i \ne 0$, and $l \ge 1$ such that
  \[
  \sum_{i=0}^n P_i\big(x^{\frac{1}{l}}\big) F(x^{k^i}) = A\big(x^{\frac{1}{l}}\big).
  \]
  Then there exist $Q_0$, \dots,~$Q_n \in K[x]$ with $Q_0 \ne 0$ and $B \in K[x]$ such that
  \[
  \sum_{i=0}^n Q_i(x) F(x^{k^i}) = B(x).
  \]
\end{lemma}

\begin{proof}
  Substitute $x^{1/l}$ for $x$ in \cref{l:mahler-eqn-l-power}.
\end{proof}

\subsection{Mahler denominators under some algebraic operations}
Recall, from the preliminaries, that a polynomial $A \in K[x]$ is $k$-negligible if all its nonzero roots $\alpha \in \algc{K}$ are roots of unity and satisfy $\alpha^{k^i} \ne \alpha$ for all $i \ge 1$.
More generally, we can compare two polynomials as follows.

\begin{definition}
  For $P$,~$Q \in K[x]$, we write
  \[
   P \preceq Q
  \] 
  if there exists a $k$-negligible $A \in K[x]$ and $s \ge 0$ such that $P$ divides
  \[
    A(x) \prod_{i=0}^s Q(x^{k^i}).
  \]
\end{definition}

The definition of $\preceq$ depends on $k$.
Since $k$ remains fixed throughout the paper, we choose not to make this dependency explicit in the notation.
Note that $\preceq$ is a reflexive and transitive relation on $K[x]$, and hence it is a preorder.
If $P \preceq Q$, then also $PR \preceq QR$ for all $R \in K[x]$ and $P(x^l) \preceq Q(x^l)$ for all $l \ge 1$.
We write $P \sim Q$ if $P \preceq Q$ and $Q \preceq P$.
Then $\sim$ is an equivalence relation on $K[x]$.

With this notation in hand, we can state the main result of the present section.

\begin{theorem} \label{t:mahler-denominators}
  Let $K$ be a field of characteristic $0$ and let $k \ge 2$.
  Let $F$,~$G \in \power{K}$ be a $k$-Mahler series and $r \ge 1$.
  Then
  \begin{enumerate}
  \item\label{md:k-powers} $\md_{k^r} F \sim \md_k F$,
  \item\label{md:sum} $\md_{k}(F+G) \preceq \md_k F \cdot \md_k G$,
  \item\label{md:root-of-unity-spec} if $\omega^k=\omega \in K \setminus \{0\}$, then $\md_k(F(\omega x))=(\md_k F)(\omega x)$,
  \item\label{md:root-of-unity} if $\omega \in \ru(K) \setminus \rucommon(K)$ then $\md_k(F(\omega x)) \sim (\md_k F)(\omega x)$,
  \item\label{md:mahler-op} 
  if $H \coloneqq \sum_{i=0}^r P_i(x) F(x^{k^i})$ with $P_i \in K[x]$, then $\md_k H \preceq (\md_k F)^r$.
  \end{enumerate}
\end{theorem}

\begin{proof}
  \ref{md:k-powers}
  We have $\md_{k} F \preceq \md_{k^r} F$, because every $k^r$-Mahler equation is also a $k$-Mahler equation.
  It remains to show $\md_{k^r} F \preceq \md_k F$.
  
  First observe that if $P \in K[x]$ is a polynomial with $P(0)=1$, then the formal infinite product
  \[
  \prod_{j=0}^\infty P(x^{k^j}) \in K\llbracket x \rrbracket
  \]
  is well-defined.
  Let $(\md_k F)(x) = x^a P_0(x)$ with $P_0 \in K[x]$ such that $P_0(0)=1$ and $a \ge 0$.
  Define $D_F \coloneqq \prod_{j=0}^\infty P_0(x^{k^j})$, and let
  \[
  x^a P_0(x) F(x) = \sum_{i=1}^t P_i(x) F(x^{k^i})
  \]
  with $P_i \in K[x]$ be a $k$-Mahler equation for $F$.
  Multiplying the equation by $\prod_{j=1}^\infty P_0(x^{k^j})$, we get 
  \[
  x^a (D_F F)(x) = \sum_{i=1}^t P_i(x) P_0(x^k) \cdots P_0(x^{k^{i-1}}) (D_F F)(x^{k^i}).
  \]
  Thus, the series $D_F F$ is $k$-Mahler with $\md_k(D_F F)$ dividing $x^a$.
  \cref{t:abs-regular} implies that $D_F F$ is $k$-regular.
  Consequently, it is also $k^r$-regular.
  Applying \cref{t:abs-regular} again, we have $\md_{k^r}(D_F F) = A_1$ with $A_1 \in K[x]$ being $k^r$-negligible, which is the same as it being $k$-negligible.
  Let $s \ge 1$ and $Q_1$, \dots,~$Q_s \in K[x]$ be such that
  \[
  A_1(x) (D_F F)(x) = \sum_{i=1}^s Q_i(x) (D_F F)(x^{k^{ri}}).
  \]
  Expanding the infinite product $D_F(x)$ yields
  \[
  A_1(x) \prod_{j=0}^\infty P_0(x^{k^j}) F(x) = \sum_{i=1}^s Q_i(x) F(x^{k^{ri}}) \prod_{j=0}^\infty P_0(x^{k^{ri+j}}) 
  \]
  We can cancel $\prod_{j=rs}^\infty P_0(x^{k^j})$ from each term, and so
  \[
  A_1(x) \prod_{j=0}^{rs-1} P_0(x^{k^{j}}) F(x) = \sum_{i=1}^s \widetilde Q_i(x) F(x^{k^{ri}})
  \]
  with some $\widetilde Q_i \in K[x]$.
  We conclude that $\md_{k^r} F$ divides $A_1(x) \prod_{j=0}^{rs-1} P_0(x^{k^{j}})$, and the claim follows because $(\md_k F)(x) = x^a P_0(x)$.

  \smallskip
  \ref{md:sum}
  Let $P_0$ and $D_F$ be as in \ref{md:k-powers} and similarly define $(\md_k G)(x) = x^b Q_0(x)$ and $D_G = \prod_{j=0}^\infty Q_0(x^{k^j})$.
  Then
  \[
  x^a P_0(x) F(x) = \sum_{i=1}^t P_i(x) F(x^{k^i}) \qquad\text{and}\qquad 
  x^b Q_0(x) G(x) = \sum_{i=1}^t Q_i(x) G(x^{k^i}),
  \]
  with suitable $P_i$, $Q_i \in K[x]$.
  Multiplying these equations by $D_F D_G / P_0$, respectively, by $D_F D_G / Q_0$, and applying \cref{t:abs-regular}, we see that $D_F D_G F$ and $D_F D_G G$ are both $k$-regular.
  Since sums of $k$-regular series are $k$-regular, also $D_F D_G (F+G)$ is $k$-regular.
  Applying once again \cref{t:abs-regular}, therefore $\md_k(D_F D_G(F+G)) = A_2$ with $A_2 \in K[x]$ being $k$-negligible.
  
  We now proceed as in \ref{md:k-powers}: let $H=D_F D_G (F+G)$ and let $R_i \in K[x]$ be such that $A_2(x)  H(x) = \sum_{i=1}^s R_i(x) H(x^{k^i})$.
  Expanding $D_F D_G$ and cancelling common factors, we find
  \[
  A_2(x) \prod_{j=0}^{s-1} (P_0 Q_0)(x^{k^j}) \cdot (F+G)(x) = \sum_{i=1}^s \widetilde R_i(x)\cdot (F+G)(x^{k^i}),
  \]
  so that $\md_k(F+G) \mid A_2(x) \prod_{j=0}^{s-1} (P_0 Q_0)(x^{k^j})$.

  \smallskip
  \ref{md:root-of-unity-spec}
  If $P_0(x) F(x) = \sum_{i=1}^n P_i(x) F(x^{k^{i}})$ with $P_0$, \dots,~$P_n \in K[x]$ and $P_0 \ne 0$, then
  \[
    P_0(\omega x) F(\omega x) = \sum_{i=1}^n P_i(\omega x) F((\omega x)^{k^{i}}) = \sum_{i=1}^n P_i(\omega x) F(\omega x^{k^i})
  \]
  is a $k$-Mahler equation for $F(\omega x)$.
  We see that $\md_{k}(F(\omega x))$ divides $(\md_{k} F)(\omega x)$.

  The converse divisibility follows by applying the statement to $F(x) = F(\omega^{-1}(\omega x))$.

  \smallskip
  \ref{md:root-of-unity}
  Let $q \in \N$ be such that $\omega^q = 1$.
  Since $\gcd(k,q)=1$, there exists $u \ge 1$ such that $k^u \equiv 1 \mod q$.
  Then $\omega^{k^u} = \omega$.
  Using \ref{md:k-powers} and \ref{md:root-of-unity-spec}, we get
  \[
  \md_k(F(\omega x)) \sim \md_{k^u}(F(\omega x)) = (\md_{k^u} F)(\omega x) \sim (\md_k F)(\omega x).
  \]

  \smallskip
  \ref{md:mahler-op}
  First note that if $P \in K[x]$ is any polynomial, then $\md_k(P F)$ divides $\md_k F$ and similarly $\md_k(F(x^{k^j}))$ divides $(\md_k F)(x^{k^j})$ for all $j \ge 1$: this follows from a $k$-Mahler equation $P_0(x) F(x) = \sum_{i=1}^n P_i(x) F(x^{k^i})$ with $P_i \in K[x]$ and $P_0=\md_k F$ by multiplying by $P$ and by substituting $x^{k^i}$ for $x$, respectively.

  These two observations, together with \ref{md:sum}, imply
  \[
  \md_k(H) \preceq (\md_k F)(x) \cdot (\md_k F)(x^k) \cdots (\md_k F)(x^{k^r}) \preceq (\md_k F)^r. \qedhere
  \]
\end{proof}

\subsection{Mahler denominators of rational functions}
As a final result in this section, we show that roots of unity with order coprime to $k$ never appear as roots of the Mahler denominator of a rational function.

\begin{proposition} \label{p:md-for-rational}
  If $F \in \power{K}$ is rational, then $\md_k F$ has no roots in $\ru(\algc{K}) \setminus \rucommon(\algc{K})$.
\end{proposition}

In particular, if $F$ is rational and all nonzero roots of $\md_k F$ are roots of unity, then $\md_k F$ is already negligible.
For the proof, we need the following lemma.

\begin{lemma} \label{l:kn-subst-multiplicity}
  If $P \in K[x]$ is a polynomial and $\alpha \in \algc{K}$ with $\alpha^{k^n}=\alpha$ for some $n \ge 1$, then the multiplicity of $\alpha$ as a root of $P$ equals the multiplicity of $\alpha$ as a root of $P(x^{k^n})$.
\end{lemma}

\begin{proof}
  Let $P(x) = (x-\alpha)^e P_0(x) \in \algc{K}[x]$ with $P_0(\alpha) \ne 0$.
  Then $P(x^{k^n}) = (x^{k^n}-\alpha)^e P_0(x^{k^n})$ and $P_0(\alpha^{k^n}) \ne 0$ since $\alpha=\alpha^{k^n}$.
  Factoring $x^{k^n} - \alpha = \prod_{i=0}^{k^n - 1} (x - \zeta^i \alpha)$ with $\zeta \in \algc{K}$ a $k^n$-th primitive root of unity, the claim follows.
\end{proof}

\begin{proof}[Proof of \cref{p:md-for-rational}]
  Let $F=P/Q$ with coprime $P$,~$Q \in K[x]$ and $Q(0) \ne 0$.
  Suppose that $\alpha \in \algc{K}$ is a root of $\md_k F$ with $\alpha^{k^n} = \alpha$ for some $n \ge 1$.

  The equality $Q(x^{k^n}) F(x^{k^n}) = P(x^{k^n})$ gives rise to a non-trivial homogeneous Mahler equation
  \begin{equation} \label{eq:mahler-rational}
  P(x^{k^n}) Q(x) F(x) = P(x) Q(x^{k^n}) F(x^{k^n}).
  \end{equation}
  \Cref{l:kn-subst-multiplicity} shows that the multiplicities of $\alpha$ in $P(x^{k^n})Q(x)$ and in $P(x) Q(x^{k^n})$ coincide.
  Dividing \eqref{eq:mahler-rational} by $\gcd(P(x^{k^n}) Q(x), P(x) Q(x^{k^n}))$ 
  yields a Mahler equation for $F$ in which $\alpha$ is not a root of the constant coefficient.
  We conclude $(\md_k F)(\alpha) \ne 0$.
\end{proof}

\section{Multiplicative Mahler series with non-completely multiplicative primes}  \label{sec:mult-mahler-non-compl-mult}

The goal of this section is to prove the following proposition, that deals with the first of the two cases of Mahler series with multiplicative coefficients.

\begin{proposition} \label{p:bad-q-mean-regular}
  Let $F = \sum_{n=1}^\infty f(n) x^n \in K\llbracket x \rrbracket$ be a $k$-Mahler series.
  If $f$ is multiplicative and there exist arbitrarily large primes such that $f(qn) \ne f(q)f(n)$ for some $n \ge 2$, then $F$ is $k$-regular.
\end{proposition}

Note that the assumption is equivalent to the existence of arbitrarily large primes $q$ such that $f(q^i) \ne f(q)f(q^{i-1})$ for some $i \ge 2$.

We start with two technical lemmas on divisibility of polynomials.
Suppose that $P \in K[x]$ and $L$ is the field obtained from the prime field of $K$ by adjoining all roots of $P$.
Then $L$ also contains all coefficients of $P$, and it can therefore be considered to be the splitting field of $P$ over the prime field of $K$.
Since $L$ is a finitely generated field, we see that $\mu(L)$ is finite.
In particular, there exist infinitely many $q \in \N$ satisfying the conditions in the following lemmas.

\begin{lemma} \label{l:q-roots}
  Let $Q$, $P \in K[x]$ with $Q \ne \lambda x^n$ for all $\lambda \in K$, $n \ge 0$.
  Let $L \subseteq \algc{K}$ be the field obtained by adjoining all roots of $P$ to the prime field of $K$.
  Suppose $q \in \N$ is such that $\gcd(2k,q)=1$ and that $\mu_q^*(L) =  \emptyset$.
  If $\omega \in \mu_q^*(\algc{K})$, then
  \[
  Q(x^{q^2}) \nmid \prod_{i=0}^e \prod_{j=0}^{q-1}  P(x^{qk^i}) P(\omega^j x^{k^i}),
  \]
  for all $e \ge 0$.
\end{lemma}

\begin{proof}
  Suppose, for sake of contradiction, that there exists an $e \ge 0$ such that $Q(x^{q^2})$ divides the right side.
  Let $\alpha \in \algc{K}$ be a nonzero root of $Q(x^{q^2})$.
  Let $\zeta \in \mu_{(q^2)^*}(\algc{K})$.
  Then $\{ \alpha, \zeta \alpha, \dots, \zeta^{q^2-1} \alpha \}$ are roots of $Q(x^{q^2})$, and hence also of the right side.

  Let $P_1(x) \coloneqq \prod_{i=0}^e P(x^{q k^i})$ and $P_2(x) \coloneqq \prod_{i=0}^e \prod_{j=1}^{q-1} P(\omega^j x^{k^i})$.
  By the pigeonhole principle, at least two of $\{ \alpha, \zeta \alpha, \zeta^2 \alpha \}$ are roots of $P_1$ or at least two of them are roots of $P_2$.
  We distinguish two cases.
  
  \smallskip
  \emph{Case 1:} There exist $l \ne l' \in \{0,1,2\}$ such that $\zeta^l \alpha$ and $\zeta^{l'} \alpha$ are roots of $P_1$.
  Then $\gamma\coloneqq \zeta^{lqk^i} \alpha^{qk^i} \in L$ and $\gamma'\coloneqq\zeta^{l'qk^{i'}} \alpha^{qk^{i'}} \in L$, for some $i$, $i' \ge 0$.
  Without restriction $i \le i'$.
  Then also
  \[
  \frac{\gamma'}{\gamma^{k^{i'-i}}} = \zeta^{q(l' - l)k^{i'}} \in L. 
  \]
  Since $l'-l \in \{\pm 1, \pm 2\}$ and $\gcd(2k,q)=1$, we see that this is a $q$-th primitive root of unity, in contradiction to the choice of $q$.

  \smallskip
  \emph{Case 2:} There exist $l \ne l' \in \{0,1,2\}$ such that $\zeta^l \alpha$ and $\zeta^{l'} \alpha$ are roots of $P_2$.
  Then there exist $i$,~$i'$, $j$,~$j'$ such that $\gamma\coloneqq \omega^{j}\zeta^{l k^i}\alpha^{k^i} \in L$ and $\gamma' \coloneqq \omega^{j'}\zeta^{l' k^{i'}} \alpha^{k^{i'}} \in L$.
  Without restriction $i \le i'$.
  Then
  \[
  \frac{\gamma'}{\gamma^{k^{i'-i}}} = \omega^{j' - j k^{i'-i}} \zeta^{(l'-l)k^{i'}} \in L.
  \]
  Again observing $l'-l \in \{\pm 1, \pm 2\}$ and $\gcd(2k,q)=1$, we have that $\zeta^{(l'-l)k^{i'}}$ is a primitive $q^2$-th root of unity, and so is $\gamma' \gamma^{-k^{i'-i}}$ (because $\omega$ is a $q$-th root of unity).
  This contradicts the choice of $q$.
\end{proof}

\begin{lemma} \label{l:coprime}
  Let $P \in K[x]$ and let $L \subseteq \algc{K}$ be the field obtained by adjoining all roots of $P$ to the prime field of $K$.
  Let $q \in \N$ with $\gcd(k,q)=1$ and $\mu_q(L) = \{1\}$.
  If $\omega \in \mu_q^*(\algc{K})$, then, for every $i$, $n \ge 0$ and $j \ge 0$ with $q \nmid j$, the polynomials
  \[
  P(x^{k^i}) \quad\text{and}\quad P(\omega^j x^{n})
  \]
  are coprime, except possibly for common roots at $0$.
\end{lemma}

\begin{proof}
  Suppose that $0 \ne \beta \in \algc{K}$ is a common root of $P(x^{k^i})$ and $P(\omega^j x^{n})$.
  Then there exist roots $\alpha_1$,~$\alpha_2$ of $P$ in $L$ such that $\alpha_1 = \beta^{k^i}$ and $\alpha_2 = \omega^j \beta^n$.
  Then $\alpha_1^{n q} = \beta^{k^i n q} = \alpha_2^{k^i q}$.
  Hence, the ratio $\alpha_1^{n} /\alpha_2^{k^i} \in L$ is a $q$-th root of unity, and our assumption on $L$ implies $\alpha_1^{n} = \alpha_2^{k^i}$.
  Now $\beta^{k^i n} = \alpha_1^{n} = \alpha_2^{k^i} = \omega^{j k^i}\beta^{n k^i}$.
  Since $\beta \ne 0$, we have $\omega^{j k^i}=1$, which is impossible because $\gcd(q,k^i)=1$ and $q \nmid j$.
\end{proof}

To a Mahler series with multiplicative coefficients and a suitable prime $q$, we can now associate a related $k$-regular series $G_q$.

\begin{lemma} \label{l:Gq_regular}
  Let $F \in K\llbracket x \rrbracket$ be a $k$-Mahler series with multiplicative coefficients.
  Let $L \subseteq \algc{K}$ be the field obtained by adjoining all roots of $\md_k F$ to the prime field of $K$.
  Let $q$ be a prime, let $\omega \in \mu_q^*(\algc{K})$, and define
  \[
  G_q(x) \coloneqq \sum_{j=0}^{q-1} F(\omega^j x) - q f(q) F(x^q).
  \]
  Then $G_q$ has coefficients in $K$.
  If $q \nmid k$ and $\mu_q^*(L) = \emptyset$, then $G_q$ is $k$-regular.
\end{lemma}

\begin{proof}
  We work over $\algc{K}$, so that $F(\omega^j x) \in \algc{K}\llbracket x \rrbracket$ is $k$-Mahler for all $j$.
  Recall that
  \begin{equation} \label{eq:sums-of-powers-ru}
  \sum_{j=0}^{q-1} \omega^{nj} =
  \begin{cases}
  q &\text{if $q \mid n$},\\
  0 &\text{if $q \nmid n$}.
  \end{cases}
  \end{equation}

  Using first \eqref{eq:sums-of-powers-ru} and then $f(qm)=f(q)f(m)$ whenever $q \nmid m$, we obtain
  \[
  G_q(x) = \sum_{m=1}^\infty \big(q f(qm) - q f(q) f(m)\big) x^{qm} = \sum_{m=1}^\infty \big(q f(q^2 m) - q f(q)f(qm)\big)x^{q^2m}.
  \]
  It follows that $G_q(x)$ is supported on $\{\, x^{q^2 m} : m \ge 1 \,\}$.

  \Cref{l:divide-support} shows $\md_{k} G_q=D(x^{q^2})$ for some polynomial $D$.
  Now \ref{md:sum} and \ref{md:root-of-unity} of \cref{t:mahler-denominators}, together with \ref{div-sup:denom} of \cref{l:divide-support}, imply 
  \begin{equation} \label{eq:Gq-md}
  \md_{k} G_q \preceq \md_{k}\big(F(x^q)\big) \cdot \prod_{j=0}^{q-1} \md_{k}\big(F(\omega^{j} x)\big) \sim (\md_{k} F)(x^q) \cdot \prod_{j=0}^{q-1}  (\md_{k} F)(\omega^{j} x).
  \end{equation}

  Suppose $0 \ne \alpha \in \algc{K}$ is a root of $\md_k G_q = D(x^{q^2})$.
  Then \cref{l:q-roots}, applied to $Q=x-\alpha^{q^2}$, shows that there exists a $\zeta \in \mu_q^*(\algc{K})$ such that $\zeta \alpha$ is not a root of
  \[
  \prod_{i=0}^{e} \prod_{j=0}^{q-1}  (\md_{k} F)(x^{qk^i}) \cdot (\md_{k} F)(\omega^{j} x^{k^i}),
  \]
  for any $e \ge 0$.
  \Cref{eq:Gq-md} therefore implies $\zeta \alpha \in \rucommon(\algc{K})$. 
  Hence, also $\alpha \in \rucommon(\algc{K})$, and $\md_k G_q$ is $k$-negligible.

  \Cref{t:abs-regular} implies that $G_q$ is $k$-regular, considered as a series over $\algc{K}$.
  Since $G_q \in K\llbracket x \rrbracket$, it is also $k$-regular over $K$ by \cref{l:regular-descent-field}.
\end{proof}

Finally, we can prove the main result of this section.

\begin{proof}[Proof of \cref{p:bad-q-mean-regular}]
  Let $q$ be a sufficiently large prime so that $G_q=\sum_{n=1}^\infty g(n)x^n$, defined as in \cref{l:Gq_regular}, is $k$-regular and such that there exists an $m \ge 2$ with $f(qm) \ne f(q)f(m)$.
  Since $f$ is multiplicative, there must exist $i \ge 2$ such that $f(q^i) \ne f(q) f(q^{i-1})$.
  Then, for every $m \ge 1$ with $q \nmid m$, we have
  \[
  g(q^i m) = q f(q^i m) - q f(q) f(q^{i-1} m) = f(m) q \big(f(q^i) - f(q) f(q^{i-1})\big) = f(m) qc,
  \]
  with $c \coloneqq f(q^i) - f(q) f(q^{i-1}) \ne 0$.

  The sequence $m \mapsto g\big(q^{i}(q m + 1)\big) = f(qm+1)qc$ is $k$-regular by \cref{l:regular-arith-progr}.
  Applying the same lemma once more, and dividing by the constant $c$, we deduce that $h\colon \N \to K$, defined by
  \[
  h(n) = \begin{cases}
    q f(n) &\text{if $n \equiv 1 \mod q$},\\
    0 &\text{if $n \nequiv 1 \mod q$},
  \end{cases}
  \]
  is $k$-regular.
  Considering the series $H(x)\coloneqq \sum_{n=1}^\infty h(n) x^n$, the Mahler denominator $\md_k H$ is therefore $k$-negligible by \cref{t:abs-regular}.

  We have 
  \begin{equation}\label{eq:H}
  H(x) = \sum_{m=1}^\infty q f(qm+1) x^{qm+1} = \sum_{j=0}^{q-1} \omega^{-j} F(\omega^j x).
  \end{equation}
  Since $\gcd(k,q)=1$, we have $\md_{k}\big( F(\omega^j x) \big) \sim (\md_{k} F)(\omega^j x)$ by \cref{t:mahler-denominators}.
  Rearranging \eqref{eq:H} gives
  \[
  \md_{k} F \preceq \prod_{j=1}^{q-1} (\md_{k} F)(\omega^{j}x). 
  \]

  However, for sufficiently large $q$, the polynomials $\md_{k} F$ and $(\md_{k} F)(\omega^j x^{k^{i}})$ are coprime for all $i \ge 0$ and $1 \le j \le q-1$ by \cref{l:coprime}.
  It follows that $\md_{k} F$ has all its roots in $\rucommon(\algc{K}) \cup \{0\}$, and therefore $F$ is $k$-regular by \cref{t:abs-regular}.
\end{proof}

\section{Rings of Mahler operators} \label{sec:rings-mahler-operators}

Before being able to deal with the second case of Mahler series with multiplicative coefficients, we need to study the ring of Mahler operators in some more detail.
The aim of the purely algebraic results in this section is to establish \cref{l:subring-intersection,l:constant-factor-mahler}.

For a ring $R$ and a ring endomorphism $\sigma$ of $R$, let $R[y;\sigma]$ denote the \defit{skew polynomial ring} (with coefficients written on the left).
Each element of $R[y;\sigma]$ has a unique representation of the form
\[
f = \sum_{i=0}^n a_i y^i,
\]
with $n \ge 0$ and $a_i \in R$, and the multiplication obeys $y^i r = \sigma^i(r)y$ for all $r \in R$.
The \defit{degree} $\deg(f)$ of $f$ is $\max\{\, i : a_i \ne 0 \,\}$ if $f \ne 0$, and $\deg(0)=-\infty$.
We recall some standard facts.

\begin{lemma} \label{l:skew-poly} Let $R$ be a ring with endomorphism $\sigma$.
  \begin{enumerate}
  \item \label{skew-poly:domain} If $R$ is a domain and $\sigma$ is injective, then $R[y;\sigma]$ is a graded domain, with grading induced by the degree.
  \item \label{skew-poly:division} If $f$,~$g \in R[y;\sigma]$ and the highest coefficient of $g$ is invertible in $R$, then there exist $q$,~$r \in R[y;\sigma]$ such that $f = qg + r$ and $\deg(r) < \deg(g)$.
  \item \label{skew-poly:pid} If $D$ is a division ring, and $\sigma$ is an endomorphism, then $D[y;\sigma]$ is a left PID.
  \end{enumerate}
\end{lemma}

\begin{proof}
  These statements are easy to check and can be found in \cite[Chapter~3.2]{berrick-keating00} or \cite[\S 1.2]{mcconnell-robson01}, but note that in both texts coefficients are written on the right.

  Specifically, statement~\ref{skew-poly:domain} follows from \cite[Lemma 3.2.5]{berrick-keating00} (with the grading being obvious), the division algorithm~\ref{skew-poly:division} is described in \cite[Section~3.2.6]{berrick-keating00} (only over division rings, but the more general claim is shown in the same way), and statement~\ref{skew-poly:pid} follows from \cite[Theorem~3.2.10]{berrick-keating00}.
\end{proof}

Every left PID $R$ is a left Ore domain and therefore admits a (classical) left quotient division ring $Q_l(R)$, the elements of which can be represented as $d^{-1}r$ with $r \in R$, $0 \ne d \in R$ \cite[Chapter~10]{lam99}.

If $\sigma$ is an injective endomorphism of $R$, then the universal property of localizations yields an extension to an endomorphism $\sigma$ of $Q_l(R)$ such that $\sigma(d^{-1}r) = \sigma(d)^{-1}\sigma(r)$.
In this way, the ring $R[y;\sigma]$ embeds into $Q_l(R)[y;\sigma]$, which in turn embeds into $Q_l(R[y;\sigma])$ (since $\sigma$ is injective and $R$ is a left Ore domain, also $R[y;\sigma]$ is a left Ore domain \cite[Theorem 10.28]{lam99}, so $Q_l(R[y;\sigma])$ indeed exists).

\begin{remark}
  In the context of Mahler operators, the left/right distinction as well as the careful consideration of quotient division rings is important.
  The ring $K(x)[y;\mahler_k]$ with $\mahler_k(x)=x^k$ ($k \ge 2$) is a left PID but not right noetherian (in fact not even right Ore, see \cite[Example~2.11(ii)]{mcconnell-robson01} for the case $k=2$, with the general case being analogous), and hence admits a left quotient division ring, but not a right quotient division ring.
  The ring $K[x][y;\mahler_k]$ is neither left nor right noetherian (see \cite[Example~2.11(iii)]{mcconnell-robson01} for $k=2$).
\end{remark}

Let $D$ be a division ring with endomorphism $\sigma$ (which is automatically injective). 
For $n \in \N$, let $D[y^n]$ be the $D$-subring of $D[y;\sigma]$ generated by $y^n$, that is, the ring of all polynomial expressions in $y^n$ with coefficients from $D$.
Since $y^n a = \sigma^n(a) y^n$ for all $a \in D$, one has $D[y^n] \cong D[z;\sigma^n]$.
In particular, the ring $D[y^n]$ is also a left PID, and therefore possesses a left quotient division ring $Q_l(D[y^n])$.
This left quotient ring $Q_l(D[y^n])$ can be canonically identified with the subring of $Q_l(D[y;\sigma])$ consisting of elements of the form $b^{-1} a$ with $a$,~$b \in D[y^n]$ and $b \ne 0$.

If $\sigma$ is an automorphism, then, by symmetry, there also exists the classical right quotient ring $Q_r(D[z;\sigma])$ and it can be canonically identified with $Q_l(D[z;\sigma])$.
To simplify the notation, we make the following definitions.

\begin{definition}
  Let $D$ be a division ring with automorphism $\sigma$.
  \begin{enumerate}
  \item By $D(y;\sigma)$ we denote the quotient division ring $Q_l(D[y;\sigma])=Q_r(D[y;\sigma])$ of the skew polynomial ring $D[y;\sigma]$.
  \item By $D(y^n) \subseteq D(y;\sigma)$ we denote the subring that is the left and right quotient division ring of the subring $D[y^n] \subseteq D[y;\sigma]$.
  \end{enumerate}
\end{definition}

We will need the following observation on intersecting infinitely many such subrings.

\begin{lemma} \label{l:subring-intersection}
  Let $D$ be a division ring and $\sigma$ an automorphism of $D$.
  Then, for every infinite subset $N \subseteq \N$,
  \[
  D = \bigcap_{n \in N} D(y^n) \ \subseteq\ D(y;\sigma).
  \]
\end{lemma}

\begin{proof}
  Clearly $D \subseteq \bigcap_{n \in N} D(y^n)$, and we have to show the converse inclusion.

  Let $a \in \bigcap_{n \in N} D(y^n)$.
  Without restriction $a \ne 0$.
  Let $a = f^{-1} g$ for some $n \in N$ and $f$,~$g \in D[y^n]$.
  Choose $m \in N$ such that $m > \max\{\deg_y(f), \deg_y(g)\}$.
  Since $a \in D(y^m)$, there exist $r$,~$s \in D[y^m]$ such that $f^{-1}g = rs^{-1}$.
  Then $gs=fr$.
  Write $r=c y^{m m_r} + y^{m(m_r+1)} \widetilde{r}$ and $s=d y^{m m_s} + y^{m(m_s+1)} \widetilde{s}$ with $m_r$, $m_s \ge 0$, with $c$,~$d \in D \setminus \{0\}$ and $\widetilde{r}$,~$\widetilde{s} \in D[y^m]$.
  The degree bounds on $g$ and $f$ imply $g d y^{m m_s} = f c y^{m m_r}$, and in turn $gd=fc$.
  We get $a=f^{-1}g=cd^{-1} \in D$.
\end{proof}

\begin{remark} \label{rem:non-surjectivity}
  The surjectivity in the previous lemma is necessary.
  Indeed, if $\sigma$ is injective but not surjective, then $D[y^n]$ is still a left PID, and one can consider the intersection of left quotient rings $\bigcap_{n \ge 1} Q_l(D[y^n])$ in $Q_l(D[y;\sigma])$.
  Let $d \in D \setminus \sigma(D)$. Since $c\coloneqq y^{-n} (\sigma^{n-1}(d)y^n) = y^{-m} (\sigma^{m-1}(d)y^m)$ for all $m$,~$n \ge 1$, the element $c$ is contained in the intersection.
  But since $y^{n} c = \sigma^{n-1}(d) y^n$ we must have $c \not\in D$, as otherwise $\sigma(c)=d$.
\end{remark}

\cref{rem:non-surjectivity} explains why we will consider Mahler operators whose coefficients are Puiseux polynomials.

\begin{definition}
  \begin{enumerate}
  \item By $\puipol{K}{x}$ we denote the ring of Puiseux polynomials with non-negative exponents and coefficients in the field $K$, that is, the ring of formal finite $K$-linear combinations of $x^q$ with $q \in \Q_{\ge 0}$.
  \item By $K(x^{\frac{1}{\infty}})$ we denote the quotient field of $\puipol{K}{x}$.
  \item By $\puiser{K}{x}$ we denote the field of Puiseux series with coefficients in $K$, that is, the ring of formal series of the form
  \[
  f=\sum_{k=k_0}^\infty a_{k/n} x^{k/n}, \qquad\text{with}\quad n \ge 1,\, k_0 \in \Z, \, a_{k/n} \in K.
  \]
  \end{enumerate}
\end{definition}

Algebraically, the ring $\puipol{K}{x}$ is the semigroup algebra of the semigroup $(\Q_{\ge 0},+)$.

Nonzero integers $k_1$, $\dots$,~$k_r \in \Z$ are \defit{multiplicatively independent} if the multiplicative subgroup of $\Q^\times$ generated by $k_1$, \dots,~$k_r$ is torsion-free of rank $r$.
Explicitly, if $k_1^{n_1}\cdots k_r^{n_r}=1$ for some $n_1$, \dots,~$n_r \in \Z$, then $n_1=\cdots=n_r = 0$.

\begin{proposition} \label{p:mahler-operators-algebraically}
  Let $k_1$, \dots, $k_r \in \Z_{\ge 2}$ be multiplicatively independent.
  \begin{enumerate}
    \item \label{mop-alg:polynomial}
    The ring of Mahler operators \[ K[x,\mahler_{k_1}, \dots, \mahler_{k_r}] \subseteq \End_K(\power{K}) \] is isomorphic to the iterated skew polynomial ring $K[x][y_1;\sigma_1]\dots[y_r;\sigma_r]$, where $\sigma_i(x)=x^{k_i}$ for all $i$, and $\sigma_i(y_j)=y_j$ for all $j < i$.

    \item \label{mop-alg:puiseux-polynomial}
    The ring of Mahler operators with non-negative Puiseux coefficients \[ K[x^{1/n}, n \ge 1, \mahler_{k_1}, \dots, \mahler_{k_r}] \subseteq \End_K\big(\puiser{K}{x}\big) \]  is isomorphic to the iterated skew polynomial ring $\puipol{K}{x}[y_1;\sigma_1]\dots[y_r;\sigma_r]$, where $\sigma_i(x)=x^{k_i}$ for all $i$, and $\sigma_i(y_j)=y_j$ for all $j < i$.
  \end{enumerate}
  
\end{proposition}

\begin{proof}
  It suffices to show \ref{mop-alg:puiseux-polynomial}, then \ref{mop-alg:polynomial} easily follows by restriction.

  Let $R\coloneqq K[x^{1/n}, n \ge 1, \mahler_{k_1}, \dots, \mahler_{k_r}] \subseteq \End_K\big(\puiser{K}{x}\big)$ denote the ring of Mahler operators with non-negative Puiseux coefficients.
  Clearly there is a surjective algebra homomorphism $\varphi\colon\puipol{K}{x}[y_1;\sigma_1]\dots[y_r;\sigma_r] \to R$, mapping $y_i$ to $\mahler_{k_i}$, and we have to show that $\varphi$ is injective.
  
  For this, it suffices to consider the action on monomials.
  Indeed, let
  \[
  f = \sum_{i_1,\dots,i_r=0}^m f_{i_1,\dots,i_r} y_1^{i_1} \cdots y_r^{i_r} \in \puipol{K}{x}[y_1;\sigma_1]\dots[y_r;\sigma_r], \quad\text{with}\quad f_{i_1,\dots,i_r} \in \puipol{K}{x}.
  \]
  Then $\varphi(f)(x^l) = \sum_{i_1,\dots,i_r=0}^m f_{i_1,\dots,i_r} x^{l k_1^{i_1} \cdots k_r^{i_r}}$.
  By multiplicative independence of $k_1$, \dots,~$k_r$, the numbers $k_1^{i_1} \cdots k_r^{i_r}$ are pairwise distinct for $0 \le i_1$, \dots,~$i_r \le m$.
  Choosing $l$ sufficiently large, namely $l > \deg(f_{i_1,\dots,i_r})$ for all $0 \le i_1$, \dots,~$i_r \le m$, we see that $\varphi(f)(x^l)=0$ implies $f_{i_1,\dots,i_r}=0$, and hence $f=0$.
\end{proof}

\begin{lemma} \label{l:constant-factor}
  Let $D$ be a division ring with commuting automorphisms $\sigma_y$, $\sigma_z$.
  \begin{enumerate}
  \item \label{constant-factor:ring} Extending $\sigma_z$ to $D[y;\sigma_y]$ by acting trivially on $y$, and similarly extending $\sigma_y$ to $D[z;\sigma_z]$, we can identify $D[y;\sigma_y][z;\sigma_z] = D[z;\sigma_z][y;\sigma_y]$. We have $yz=zy$.
  \item \label{constant-factor:claim} Let $f =\sum_{i=0}^n f_i z^i \in D[y;\sigma_y][z;\sigma_z]$ with $f_i \in D[y;\sigma_y]$ and $\sigma_z$ extended to act trivially on $y$.
  If $f \in D(z;\sigma_z)[y;\sigma_y] g$ for some $g \in D[y;\sigma_y]$, then $f_0 \in D[y;\sigma_y] g$.
  \end{enumerate}
\end{lemma}

\begin{proof}
  \ref{constant-factor:ring} This follows since $\sigma_y$ and $\sigma_z$ commute.

  \smallskip
  \ref{constant-factor:claim}
  By assumption $af=bg$ with $a \in D[z;\sigma_z]$ and $b \in D[z;\sigma_z][y;\sigma_y]$.
  Let $a = \sum_{i=0}^n a_i z^i$ with $a_i \in D$, and $b = \sum_{i=0}^m b_i z^i$ with $b_i \in D[y;\sigma_y]$.

  Let $j \ge 0$ be minimal with $a_j \ne 0$.
  Since $g$ is constant in $z$, a comparison of coefficients of $z^j$ in $af=bg$ shows that $a_j z^j f_0 = b_j z^j g$.
  It follows that $f_0 = z^{-j} (a_j^{-1}b_j) z^{j} g = \sigma_z^{-j}(a_j^{-1} b_j) g \in D[y;\sigma_y]g$.
\end{proof}

We will later apply the proposition with two multiplicatively independent Mahler operators $y=\mahler_k$ and $z=\mahler_l$ and $D=K(x^{\frac{1}{\infty}})$.
For ease of reference, we also record this special case.

\begin{lemma} \label{l:constant-factor-mahler}
  Let $k$,~$l \ge 2$ be multiplicatively independent.
  Let
  \[
  \cL = \sum_{i=0}^n \cL_i \mahler_l^i \ \in\  K(x^{\frac{1}{\infty}})[\mahler_k,\mahler_l],
  \]
  with $\cL_i \in K(x^{\frac{1}{\infty}})[\mahler_k]$.
  If $\cL \in K(x^{\frac{1}{\infty}},\mahler_l)[\mahler_k] \cM$ for some $\cM \in K(x^{\frac{1}{\infty}})[\mahler_k]$, then $\cL_0 \in K(x^{\frac{1}{\infty}})[\mahler_k] \cM$.
\end{lemma}

\section{Mahler series with completely multiplicative primes.} \label{sec:mult-mahler-completely-multiplicative}

The goal of this section is to prove the following theorem, which allows us to deal with the second case of Mahler series with multiplicative coefficients (the first one has been dealt with in \cref{sec:mult-mahler-non-compl-mult}).

\begin{theorem} \label{t:special-mult-implies-regular}
  Let $K$ be a field of characteristic $0$, let $k \ge 2$, and let $F = \sum_{n=1}^\infty f(n) x^n \in K \llbracket x \rrbracket$ be a $k$-Mahler series.
  Suppose that there exists $q \in \Z_{\ge 3}$ such that
  \begin{itemize}
  \item $f(qn)=f(q)f(n)$ for all $n \in \N$;
  \item for the splitting field $L$ of $\md_k F$ over the prime field of $K$ it holds that $\mu_q(L) = \{1 \}$,
  \item $\gcd(q,k)=1$,
  \end{itemize}
  then $F$ is $k$-regular.
\end{theorem}

Note that we assume $f(qn)=f(q)f(n)$ for all $n$, not just those with $q \nmid n$.
On the other hand, we assume this property only for the fixed prime $q$, that is, there is no assumption for $f$ to be multiplicative.

The skew polynomial ring $K[x^{\frac{1}{\infty}}][y;\sigma^n]$ with $\sigma(x)=x^q$ faithfully acts on $\puiser{K}{x}$ from the left by Mahler operators by \cref{p:mahler-operators-algebraically}.
We identify
\[
K[x^{\frac{1}{\infty}}][y;\sigma^n] \cong K[x^{\frac{1}{\infty}},\mahler_q^n] \ \subseteq\  \End\big(\puiser{K}{x}\big).
\]
By $K(x^{\frac{1}{\infty}},\mahler_q^n)$ we denote the quotient division ring of $K[x^{\frac{1}{\infty}},\mahler_q^n]$.

The division rings $K(x^{\frac{1}{\infty}},\mahler_q^n)$ no longer embed into $\End\big(\puiser{K}{x}\big)$.
However, as discussed in \cref{sec:rings-mahler-operators}, they are all naturally subrings of $K(x^{\frac{1}{\infty}},\mahler_q)$.

Let $F \in \power{K}$ be $k$-Mahler. A fortiori there exists $0 \ne \cM \in K(x^{\tfrac{1}{\infty}})[\mahler_k]$ such that $\cM F$ is rational. 
Since $K(x^{\tfrac{1}{\infty}})[\mahler_k]$ is a PID, there exists such an operator of minimal degree, and any other such operator is a multiple of it.
\Cref{l:frac-mahler-eqn} shows that we can actually take the operator in $K(x)[\mahler_k]$.
Normalizing it further, to be in $K[x,\mahler_k]$ with coprime polynomials as coefficients, this operator becomes unique, up to scalars, and is a \defit{minimal inhomogeneous $k$-Mahler operator} of $F$.

\begin{proposition} \label{p:annihilator-gen}
  Let $k$,~$q \ge 2$ be multiplicatively independent.
  Let $F \in K\llbracket x \rrbracket$ be a $k$-Mahler series and $\cM \in K[x,\mahler_k]$ a minimal inhomogeneous $k$-Mahler operator of $F$.

  For all $n \ge 1$, let
  \[
  I_n\coloneqq \ann_{K(x^{\frac{1}{\infty}})[\mahler_q^n,\mahler_k]}(F) \coloneqq \big\{\, \mathcal L \in K(x^{\frac{1}{\infty}})[\mahler_q^n,\mahler_k] : \mathcal L F = 0 \,\big\},
  \]
  and let $\widetilde{I}_n \coloneqq K(x^{\frac{1}{\infty}},\mahler_q^n)[\mahler_k] I_n$ be the extension of the left ideal $I_n$ to $K(x^{\frac{1}{\infty}},\mahler_q^n)[\mahler_k]$.

  Then there exists $d \ge 1$ such that $\widetilde{I}_{dn} = K(x^{\frac{1}{\infty}},\mahler_q^{dn})[\mahler_k] \cM$ for all $n \ge 1$.
\end{proposition}

\begin{proof}
  For all $n \ge 1$, let $D_n \coloneqq K(x^{\frac{1}{\infty}},\mahler_q^{n})$.
  Keep in mind that $\mahler_q^n=\mahler_{q^n}$. 
  Since $D_n$ is a division ring, the ring $D_n[\mahler_k]$ is a PID (by \cref{p:mahler-operators-algebraically} and \cref{l:skew-poly}).
  Thus, each left ideal $\widetilde{I}_n$ is principal, say $\widetilde{I}_n = D_n[\mahler_k] \cG_n$ with a monic $\cG_n \in D_n[\mahler_k]$.
  Since $\cM F$ is rational, there exists $\cL \in K[x,\mahler_k]$ with $\deg_{\mahler_k}(\cL) = 1$ such that $\cL \cM F = 0$ (\cref{exm:rational-is-mahler}).
  Since $\cL \cM \in \widetilde{I}_n$ for every $n$, we have that $\deg_{\mahler_k}(\cG_n) \le \deg_{\mahler_k}(\cM)+1$ is uniformly bounded for all $n$.
  
  If $m \mid n$, then $D_n \subseteq D_m$ and $I_n \subseteq I_m$, so in this case $\widetilde{I}_n \subseteq \widetilde{I}_m$ and therefore $\deg(\cG_m) \le \deg(\cG_n)$.
  Let $s \coloneqq \max\{\, \deg(\cG_n) : n \ge 1 \,\}$.
  If $d \ge 1$ is such that $\deg(\cG_d)=s$, then $\deg(\cG_{dn}) =s$ for all $n \ge 1$.

  We now show, in turn,
  \begin{enumerate}[label=(\roman*)]
  \item\label{ann:step1} $\cG_d = \cG_{dn}$ for all $n \ge 1$;
  \item\label{ann:step2} $\cG\coloneqq \cG_d \in K(x^{\frac{1}{\infty}})[\mahler_k]$ (that is, the coefficients of $\cG$ do not depend on $\mahler_q^d$);
  \item\label{ann:step3} $\cG F$ is rational; and
  \item\label{ann:step4} $D_{dn}[\mahler_k] \cM = D_{dn}[\mahler_k] \cG = \widetilde{I}_{dn}$ for all $n \ge 1$.
  \end{enumerate}

  \ref{ann:step1}
  Let $\cG_{dn} = \mahler_k^s + \sum_{i=0}^{s-1} g_{n,i} \mahler_k^i$ with $g_{n,i} \in D_{dn}$.
  Consider $\cH \coloneqq \cG_{dn} - \cG_{d} \subseteq D_d[\mahler_k]$.
  Then $\cH \in \widetilde{I}_d$ and $\deg_{\mahler_k}(\cH) < s$.
  By choice of $s$ and $d$, this implies $\cH = 0$.
  Thus, we have $\cG_{d}=\cG_{dn}$ for all $n \ge 1$. 

  \ref{ann:step2} The coefficients of $\cG_d=\cG_{dn}$ are contained in
  \[
  \bigcap_{n\ge 1} K(x^{\frac{1}{\infty}}, \mahler_q^{dn})= \bigcap_{n\ge 1} K(x^{\frac{1}{\infty}})(\mahler_q^{dn}).
  \]
  \Cref{l:subring-intersection} shows that this intersection is $K(x^{\frac{1}{\infty}})$.
  Thus, we have $\cG \in K(x^{\frac{1}{\infty}})[\mahler_k]$.

  \ref{ann:step3} 
  Since $\cL \cM \in \widetilde{I}_{dn}$ for all $n \ge 1$, we can write $\cL \cM = \cH_n \cG$ with $\cH_n \in D_{dn}[\mahler_k]$.
  Using again $\bigcap_{n \ge 1} D_{dn}=K(x^{\frac{1}{\infty}})$, it follows that $\cH_n \in K(x^{\frac{1}{\infty}})[\mahler_k]$.
  Since $\cH_n \cG F = \cL \cM F = 0$, \cref{l:frac-mahler-eqn} shows that $\cG F$ is $k$-Mahler.

  Starting from $\cG \in D_d[\mahler_k] I_d$, we can clear denominators to write $0 \ne \cA \cG = \sum_{i=0}^m \cB_i \cX_i$ with $\cX_i \in I_d$, $\cA \in K[x^{\frac{1}{\infty}},\mahler_q^d]$ and $\cB_i \in K[x^{\frac{1}{\infty}},\mahler_q^d, \mahler_k]$.
  Since each $\cX_i$ annihilates $F$, so does $\cA \cG$.
  As $\cA$ does not involve $\mahler_k$, \cref{l:frac-mahler-eqn} shows that $\cG F$ is $q$-Mahler.

  Since $k$ and $q$ are multiplicatively independent, Adamczewski and Bell's generalization of Cobham's Theorem to Mahler series (\cref{t:cobham-mahler}) implies that $\cG F$ is rational. 

  \ref{ann:step4}
  By definition of $\cM$, the operator $\cM$ divides $\cG$ in $K(x^{\frac{1}{\infty}})[\mahler_k]$.
  Since $\cM F$ is rational, we can easily find a nonzero $\cL'\coloneqq P(x) - Q(x) \mahler_q^d$ such that $\cL'\cM F = 0$ (\cref{exm:rational-is-mahler}).
  Thus, it holds that $\cL' \cM \in I_{d}$. But in $D_d[\mahler_k]$ now $\cL'$ is invertible, so that $\cM \in \widetilde{I}_d$.
  This forces $D_{dn}[\mahler_k] \cM = D_{dn}[\mahler_k] \cG$, and $\widetilde{I}_{dn} = D_{dn}[\mahler_k] \cG$ holds by definition of $\cG$.
\end{proof}

The observation in the previous proof that the intersection of $D_{dn}$ is $K(x^{\frac{1}{\infty}})$ is in fact non-trivial and hinges on the fact that the Mahler operators are bijective on $K(x^{\frac{1}{\infty}})$.
See \cref{rem:non-surjectivity}.

\begin{corollary} \label{c:stabilized-mahler-op-constant}
  Suppose $k$,~$q \ge 2$ are multiplicatively independent.
  Let $F \in \power{K}$ be a $k$-Mahler series.
  Then there exists an $n_0$, such that for all $n \ge n_0$, if
  \[
  \cL=\sum_{i=0}^N \cL_i \mahler_{q}^{i\cdot n!} \in K[x^{\frac{1}{\infty}}, \mahler_k, \mahler_q^{n!}] \qquad\text{with}\qquad \cL_i \in K[x^{\frac{1}{\infty}}, \mahler_k]
  \]
  is such that $\cL F = 0$, then $\cL_0 F \in K(x^{\frac{1}{\infty}})$.
\end{corollary}

\begin{proof}
  As in \cref{p:annihilator-gen}, for $n \ge 1$, let
  \[
  I_n \coloneqq \ann_{K(x^{\frac{1}{\infty}})[\mahler_{q^n},\mahler_k]}(F) \quad\text{and}\quad \widetilde{I}_n \coloneqq K(x^{\frac{1}{\infty}},\mahler_{q^n})[\mahler_k] I_n.
  \]
  Let $\cM \in K[x,\mahler_k]$ be a minimal inhomogeneous $k$-Mahler operator of $F$.
  By \cref{p:annihilator-gen}, there exists $d \ge 1$ such that $\widetilde I_{dn} = K(x^{\frac{1}{\infty}},\mahler_{q^{dn}})[\mahler_k] \cM$ for all $n \ge 1$.

  Define $n_0 \coloneqq d$ and consider $n \ge n_0$, so that $d \mid n!$.
  Since $\cL F =0$, we have $\cL \in I_{n!} \subseteq \widetilde{I}_{n!}$.
  This means $\cL \in K(x^{\frac{1}{\infty}},\mahler_{q}^{n!})[\mahler_k]\cM$.
  \Cref{l:constant-factor-mahler} (with $l=q^{n!}$) shows $\cL_0 \in K(x^{\frac{1}{\infty}})[\mahler_k] \cM$.
  But then $\cL_0 F \in K(x^{\frac{1}{\infty}})$ by definition of $\cM$.
\end{proof}

\begin{lemma} \label{l:unit-root-avg}
  Let $\omega \in \mu_q^*(\algc{K})$.
  If $F$ and $q$ are as in \cref{t:special-mult-implies-regular}, then
  \begin{equation} \label{eq:unit-root-avg}
  \sum_{j=0}^{q-1} F(\omega^j x) = q f(q) F(x^q).
  \end{equation}
\end{lemma}

\begin{proof}
  \Cref{eq:unit-root-avg} follows from
  \[
  \sum_{j=0}^{q-1} F(\omega^j x) = \sum_{n=1}^\infty f(n) \sum_{j=0}^{q-1} \omega^{nj} x^{n} = \sum_{n=1}^\infty f(qn) q x^{nq} = \sum_{n=1}^\infty f(q)f(n) q x^{nq} = f(q) q F(x^q). \qedhere
  \]
\end{proof}

The following is a first step towards $k$-regularity of $F$.

\begin{lemma} \label{l:close-to-regular}
  If $F$ and $q$ are as in \cref{t:special-mult-implies-regular} and $K$ is algebraically closed, then all roots of $\md_k F$ are contained in $\ru(\algc{K}) \cup \{0\}$.
\end{lemma}

\begin{proof}
  Let $\omega \in \mu_q^*(K)$.
  We use the assumption $K=\algc{K}$ to ensure $F(\omega^j x)$ is also a $k$-Mahler series over $K$.

  Rearranging \eqref{eq:unit-root-avg}, we find $F(x) = - \sum_{j=1}^{q-1} F(\omega^j x) + f(q) q F(x^q)$. Iterating, we obtain
  \[
  F(x) = - \sum_{i=0}^{N-1} \sum_{j=1}^{q-1} f(q)^i q^i F(\omega^j x^{q^i}) + f(q)^N q^N F(x^{q^N})
  \]
  for all $N \ge 1$.
  Let $P= \md_k F$.
  \Cref{t:mahler-denominators} and \cref{l:divide-support} show that $P$ divides
  \[
  \prod_{i,s=0}^e \prod_{j=1}^{q-1} P(\omega^j x^{q^i k^s}) \cdot \prod_{s=0}^e P(x^{k^s q^N}) \cdot A_1(x) 
  \]
  for some $e \ge 0$ and with $A_1$ being $k$-negligible.
  
  \Cref{l:coprime} shows that $P$ is coprime to $P(\omega^j x^{q^i k^s})$ for all $i$, $s$ and $j$ with $q \nmid j$ (except possibly for common roots at $0$).
  Thus, any root $\alpha$ of $P$ that is not in $\rucommon(\algc{K}) \cup \{0\}$ must have the property that for every $N \ge 1$ there exists an $s \ge 0$ such that also $P(\alpha^{k^{s} q^N}) = 0$.
  Since $P$ has only finitely many roots, the pigeonhole principle implies that there exist $N \ne N'$ and $s$,~$s'$ such that $\alpha^{k^s q^N} = \alpha^{k^{s'} q^{N'}}$.
  We get $\alpha^{k^{s'} q^{N'} - k^s q^N}=1$.
  Since $\gcd(k,q)=1$, we have $k^{s'} q^{N'} - k^s q^N \ne 0$, and hence $\alpha \in \ru(\algc{K})$.
\end{proof}

To obtain $k$-regularity of $F$, we need to show the stronger conclusion that all roots of $\md_k F$ are in fact contained in $\rucommon(\algc{K}) \cup \{0\}$, that is, to additionally rule out any roots in $\ru(\algc{K}) \setminus \rucommon(\algc{K})$.
We now show this trickier conclusion.

\begin{proof}[Proof of \cref{t:special-mult-implies-regular}]
  Without restriction, we may assume that $K$ is algebraically closed.
  First note that $q^m$ for $m \ge 1$ also satisfies the conditions imposed on $q$ in the statement of the theorem.
  Take $m \ge 1$ sufficiently large so that \cref{c:stabilized-mahler-op-constant} applies with $n_0=1$ for $q'=q^m$.
  Let $\omega \in \mu_{q'}^*(\algc{K})$.

  Define
  \[
  G(x)\coloneqq \sum_{j=1}^{q'-1} F(\omega x) = f(q') q' F(x^{q'}) - F(x) \in K\llbracket x \rrbracket,
  \]
  with the claimed equality holding by \cref{l:unit-root-avg}. 
  The series $G$ is also $k$-Mahler, because $K$-linear combinations of $k$-Mahler series are $k$-Mahler.
  \Cref{t:mahler-denominators} implies
  \[
  \md_{k} G \ \preceq\ \prod_{j=1}^{q'-1} (\md_k F)(\omega^{j} x).
  \]

  Recalling the definition of $\preceq$, \cref{l:coprime} shows that any common factor of $\md_{k} F$ and $\md_{k} G$ must have all its roots in $\rucommon(K) \cup \{0\}$.
  In other words, we have $\gcd(\md_k F, \md_k G) \preceq 1$.

  Let
  \[
  \cL_0 = \md_{k} G - \sum_{i=1}^l P_i(x) \mahler_{k^{i}} 
  \]
  with $l \ge 0$ and polynomials $P_1$, \dots,~$P_l \in K[x]$ be such that $\cL_0 G = 0$.
  Since $\cL_0 F = f(q')q' \cL_0\mahler_{q'} F$ by definition of $G$, it follows that
  \[
  (\cL_0 - f(q')q' \cL_0 \mahler_{q'}) F = 0.
  \]

  \Cref{c:stabilized-mahler-op-constant} now implies $\cL_0 F \in K(x^{\frac{1}{\infty}})$.
  Since $\cL_0$ does not involve fractional coefficients, even $\cL_0 F \in K(x)$, say $\cL_0 F = P/Q$ with coprime $P$,~$Q \in K[x]$.
  \Cref{l:close-to-regular} shows that $\md_k F$ has all roots in $\ru(K) \cup \{0\}$.
  Statement~\ref{md:mahler-op} of \cref{t:mahler-denominators} then implies that the same is true for $\md_k(\cL_0 F)=\md_k(P/Q)$.
  The $k$-Mahler denominator of a rational function never has roots in $\ru(K) \setminus \rucommon(K)$ by \cref{p:md-for-rational}, and hence $\md_k (\cL_0 F) \in \rucommon(K) \cup \{0\}$.
  Thus, there exists a $k$-Mahler operator $\cL_1 = Q_0(x) + Q_1(x) \mahler_k$ with $\cL_1 \cL_0 F =0$ and all roots of $Q_0$ in $\rucommon(K) \cup \{0\}$.

  Now $\md_k F$ divides $Q_0 \cdot \md_k G$, but also $\gcd(\md_k F,\md_k G) \preceq 1$.
  Thus, all roots of $\md_k F$ are in $\rucommon(K) \cup \{0\}$.
  \Cref{t:abs-regular} implies that $F$ is $k$-regular.
\end{proof}

\section{Finishing the proof}

In this final section, we put everything together to prove the main result of the paper.

\begin{proof}[Proof of~\cref{t:mult-mahler-main}]
  Let $F = \sum_{n=1}^\infty f(n) x^n \in \power{K}$ be a $k$-Mahler series with multiplicative coefficient sequence $f$.
  We first show that $F$ is $k$-regular.

  Without restriction, we may assume that the field $K$ is finitely generated over its prime field and that $\md_k F$ splits over $K$.
  In particular, the set of roots of unity in $K$ is finite.

  We distinguish two cases.
  Suppose first that for all arbitrarily large primes $q$ there exists $n \ge 2$ with $f(qn) \ne f(q)f(n)$.
  Then \cref{p:bad-q-mean-regular} shows that $F$ is $k$-regular.
  
  In the second case, there exist arbitrarily large primes $q$ such that $f(qn)=f(q)f(n)$ for all $n \ge 1$.
  Since $K$ is finitely generated, we can choose such a prime $q$ so that $\mu_q(K)=\{1\}$ and such that $\gcd(q,k)=1$.
  \Cref{t:special-mult-implies-regular} implies that $F$ is $k$-regular.

  Now that we know that $F$ is $k$-regular, \cref{p:explicit-form} shows the existence of $p$, $g$, $r$, and $\chi$ as in the statement.
\end{proof}

\begin{proof}[Proof of \cref{c:mult-mahler-nonprime}]
  \Cref{t:mult-mahler-main} shows that $F$ is $k$-regular.
  Now \cref{p:regular-distinct-primes} implies that $F$ is rational.
\end{proof}

There is one final observation about multiplicative eventually periodic functions, that we have mentioned in the introduction.

\begin{lemma} \label{l:mult-eventually-periodic}
  If $\chi \colon \N \to K$ is a multiplicative eventually periodic function, then it is either periodic or eventually zero.
\end{lemma}

\begin{proof}
  Let $N$, $d \ge 1$ be such that $\chi(n+d) = \chi(n)$ for all $n \ge N$.
  Suppose that $\chi$ is not eventually zero.
  By eventual periodicity, there exists an $r$ such that $\chi(r+md) \ne 0$ for all $m \ge 0$.
  Let $0 \le n < N$.
  Dirichlet's theorem on primes in arithmetic progressions implies that there exists a prime $p > N$ with $p \equiv r \mod d$. 
  Then
  \[
  \chi(p)\chi(n) = \chi(pn) = \chi(pn+pd) = \chi(p)\chi(n+d)
  \]
  implies $\chi(n)=\chi(n+d)$, since $\chi(p) \ne 0$.
\end{proof}

\begin{singlespace}
  \printbibliography
\end{singlespace}

\end{document}